\documentclass{amsart}
\usepackage{amsmath,amsthm,amssymb}
\usepackage{tikz}
\usepackage{hyperref}
\usepackage{graphicx,color}
\numberwithin{equation}{section}
\usetikzlibrary{snakes}

\newcommand{\MM}{\mathcal{M}}

\newtheorem{thm}{Theorem}[section]
\newtheorem{lem}[thm]{Lemma}
\newtheorem{prop}[thm]{Proposition}
\newtheorem{cor}[thm]{Corollary}
\theoremstyle{definition}
\newtheorem{exam}[thm]{Example}
\newtheorem{defn}[thm]{Definition}

\newtheorem{remark}[thm]{Remark}

\newcommand\aftersubsection{\hfill\vspace{5pt}}

\newcommand\LL{\mathcal{L}}
\newcommand\FF{\mathcal{F}}
\newcommand\la{\lambda}

\newcommand\Mot{\operatorname{Motz}}
\newcommand\wt{\operatorname{wt}}
\newcommand\FT{\operatorname{FT}}

\newcommand\add{\mathsf{add}}
\newcommand\rem{\mathsf{remove}}

\newcommand\wLL{\LL^*}

\newcommand\vP{P^\vee}

\newcommand\va{a^\vee}
\newcommand\vb{b^\vee}

\newcommand\vwt{\wt^\vee}

\newcommand\Sch{\operatorname{Sch}}
\newcommand\LH{\operatorname{LH}}
\newcommand\MH{\operatorname{MH}}
\newcommand\MS{\operatorname{MS}}
\newcommand\Sym{\mathfrak{S}}
\newcommand\Fix{\operatorname{Fix}}
\newcommand\Span{\operatorname{span}}
\newcommand\cyc{\operatorname{cyc}}
\newcommand\nexc{\operatorname{nexc}}

\newcommand\Hyper[5]{
  \:{}_{#1}F_{#2} \left( \left.
    \begin{matrix}
      #3\\
      #4\\
    \end{matrix}
    \:\right|\: #5
    \right)
}
\newcommand\qHyper[5]{
  \:{}_{#1}\phi_{#2} \left( \left.
    \begin{matrix}
      #3\\
      #4\\
    \end{matrix}
    \:\right|\: #5
    \right)
}

\newcommand\bd{\mathsf{bd}}
\newcommand\bm{\mathsf{bm}}
\newcommand\rd{\mathsf{rd}}
\renewcommand\rm{\mathsf{rm}}

\def\addable(#1,#2){\draw [blue, line width=1pt] (#1,#2) circle [radius=6pt];}
\def\remove(#1,#2){\fill [red] (#1,#2) circle [radius=4pt];}
\def\BM#1{\draw [line width=2pt] (#1+0.9,0.9) rectangle +(-0.8,-0.8);} 
\def\RM#1{\draw [line width=2pt,red] (#1+0.9,0.9) rectangle +(-0.8,-0.8);}
\def\BD#1{\draw [line width=2pt] (#1+0.9,0.9) rectangle +(-1.8,-0.8);}
\def\RD#1{\draw [line width=2pt,red] (#1+0.9,0.9) rectangle +(-1.8,-0.8);}
\def\LBM#1{\BM{#1} \node at (#1+0.5, 1.3) {$1$};}
\def\LRM#1{\RM{#1} \node at (#1+0.5, 1.3) {$-b_{#1}$};}
\def\LBD#1{\BD{#1} \node at (#1, 1.3) {$-a_{#1}$};}
\def\LRD#1{\RD{#1} \node at (#1, 1.3) {$-\lambda_{#1}$};}

\title{Combinatorics of orthogonal polynomials of type $R_I$}
\author{Jang Soo Kim and Dennis Stanton}
\thanks{The first author was supported by NRF grants \#2019R1F1A1059081 and \#2016R1A5A1008055.} 
\address{Department of Mathematics,
Sungkyunkwan University (SKKU), Suwon, Gyeonggi-do 16419, South Korea}
\email{jangsookim@skku.edu}
\address{School of Mathematics,
University of Minnesota,
Minneapolis, Minnesota 55455, USA}
\email{stanton@math.umn.edu}
\date{\today}

\begin{document}

\begin{abstract} A combinatorial theory for type $R_I$ orthogonal polynomials is
  given. The ingredients include weighted generalized Motzkin paths, moments,
  continued fractions, determinants, and histories. Several explicit examples in
  the Askey scheme are given.
\end{abstract}

\dedicatory{Dedicated to Dick Askey, our great mentor.}
\maketitle
\tableofcontents

\section{Introduction}

Ismail and Masson \cite{IsmailMasson} defined orthogonal polynomials of type
$R_I$ by generalizing the three-term recurrence relation that classical
orthogonal polynomials satisfy. In this paper we develop the combinatorial
theory of type $R_I$ orthogonal polynomials. This theory parallels the
Flajolet--Viennot development \cite{Flajolet1980,ViennotLN}, for
classical orthogonal polynomials, using weighted paths. We also give type $R_I$
versions of several classical polynomials in the Askey scheme.

The orthogonal polynomials $P_n(x)$ of type $R_I$  are defined recursively.
Let $\{P_n(x)\}_{n\ge0}$ be defined by $P_{-1}(x)=0$, $P_0(x)=1$ and for $n\ge0$,
\begin{equation}
  \label{eq:favard1}
  P_{n+1}(x) = (x-b_n)P_n(x) -(a_n x+\lambda_n )P_{n-1}(x).  
\end{equation}
This defines $P_n(x)$ as a monic polynomial in $x$ of degree $n$ and also a
polynomial in the recurrence coefficients, $\{b_k\}_{k\ge0}$, $\{a_k\}_{k\ge0}$,
and $\{\lambda_k\}_{k\ge0},$ where the values \( a_0 \) and \( \lambda_0 \) are
irrelevant. Let
$$
d_m(x) := \prod_{i=1}^m (a_i x+\lambda_i), \quad and \quad Q_m(x) := \frac{P_m(x)}{d_m(x)}.
$$
The orthogonality relation for $P_n(x)$ is defined by a linear functional $\LL$ on a
certain vector space of rational functions:
$$
\LL(P_n(x)Q_m(x))=0, \quad \mbox{if} \quad 0\le n<m.
$$

The motivation for this paper is to explain this orthogonality combinatorially. 
Our combinatorial models for $\LL$ and $P_n(x)$ have weighted objects, where the 
weights depend on the coefficients $b_n$, $a_n$, and $\lambda_n$. 
Note that $a_n=0$ is the classical orthogonal polynomial case. 
In this case the objects of weight $0$
may be deleted, and Viennot's theory is recovered. 

\textbf{Throughout this paper we assume that $a_n\ne0$ and 
$P_n(-\lambda_n/a_n)\ne0$ for all $n\ge0$ unless otherwise stated.}

In the classical theory, the orthogonality takes place via a linear functional
$\LL$ on the vector space of polynomials. For type $R_I$ orthogonality, we will
need to extend the definition of $\LL$ from the space of polynomials to a larger
vector space
$$
V=\mathrm{span}\{x^n Q_m(x): n,m\ge0\}
$$ 
of rational functions.

In Section~\ref{sec:proof-orthogonality} we find several bases of the vector
space $V$. We show that there is a unique
linear functional $\LL$ on $V$ with respect to which the type $R_I$ orthogonal
polynomials $P_n(x)$ are orthogonal. This is a slight improvement of a result of
Ismail and Masson \cite{IsmailMasson}.

In Section~\ref{sec:moments-ops-type} we give combinatorial interpretations for
$\mu_n=\LL(x^n)$, $\mu_{n,m}=\LL(x^n Q_m(x))$, $\mu_{n,m,\ell}=\LL(x^n P_m(x)
Q_\ell(x))$, and $\rho_{n,m,\ell}=\LL(x^n P_m(x) P_\ell(x))$ in terms of lattice
paths called Motzkin-Schr\"oder paths. We find the infinite continued fraction
for the moment generating function. We also give a recursive formula for
$\nu_{n,m} = \LL\left(x^n/d_m(x)\right)$.

In Section~\ref{sec:inverted-polynomials} we study Laurent biorthogonal
polynomials $P_n(x)$, which are type $R_I$ orthogonal polynomials with
$\lambda_n=0$. Kamioka \cite{Kamioka2014} combinatorially studied these
polynomials. In this case there is another linear functional $\FF$ that gives a
different type of orthogonality for $P_n(x)$. We find a simple connection
between our linear functional $\LL$ and the other linear functional $\FF$ using
inverted polynomials. We then review Kamioka's results and generalize them.

In Section~\ref{sec:lattice-paths-with} we express the generating function for
Motzkin-Schr\"oder paths with bounded height in terms of inverted polynomials
where the indices of the sequences $a_n,b_n$, and $\lambda_n$ are shifted. There 
are finite continued fractions for these rational functions which are explicitly 
given by the three-term recurrence coefficients.

In Section~\ref{sec:hank-determ-type} we find determinant formulas for $P_n(x)$
and $Q_n(x)$ using $\nu_{i,j}$. We also consider some Hankel determinants using 
$\mu_n$.

In Sections~\ref{sec:explicit-type-r_i}, \ref{glue}, and \ref{powers} 
we give examples of type $R_I$ orthogonal
polynomials, including Askey--Wilson and $q$-Racah polynomials.

In Section~\ref{sec:combinatorics} we study combinatorial aspects of some
type $R_I$ polynomials in the previous sections.

\section{The orthogonality of type $R_I$ polynomials}
\label{sec:proof-orthogonality}

In this section we find several bases of the vector space 
\begin{equation}
  \label{eq:V}
  V: = \Span\{x^n Q_m(x):n,m\ge0\}.
\end{equation}
We then give a detailed proof of the following result of Ismail and Masson
\cite{IsmailMasson}.

\begin{thm}\cite[Theorem~2.1]{IsmailMasson}
  \label{thm:unique L}
  There is a unique linear functional $\LL$ on $V$ satisfying $\LL(1)=1$ and 
  \[
    \LL(x^n Q_m(x)) = 0  \qquad \mbox{if $0\le n< m$}.
  \]
\end{thm}

We note that Theorem~\ref{thm:unique L} is stated slightly differently in
\cite[Theorem~2.1]{IsmailMasson}. Instead of stating the uniqueness of $\LL$,
they write that the values of $\LL$ on the elements in the set
$\{x^n:n\ge0\}\cup\{1/d_m(x):m\ge 1\}$ are uniquely determined. We will show
that this set is a basis of $V$, hence the two statements are equivalent. In the
proof of \cite[Theorem~2.1]{IsmailMasson} they implicitly use the fact that
$\{x^{m-1}Q_{m}(x):m\ge2\}\cup \{Q_m(x):m\ge0\}$ is a basis of $V$ without a proof.
In this section we give a detailed proof of Theorem~\ref{thm:unique L} by
showing this fact.

Observe that dividing both sides of \eqref{eq:favard1} by $d_n(x)$ gives
\begin{equation}
  \label{eq:favard2}
  (a_{n+1}x+\lambda_{n+1}) Q_{n+1}(x) = (x-b_n)Q_n(x) -Q_{n-1}(x).  
\end{equation}

Let
\begin{equation}
  \label{eq:V'}
   V':= \Span\{x^n/d_m(x):n,m\ge0 \}.  
\end{equation}
We first find a basis of $V'$ and then show that $V=V'$.

\begin{lem}\label{lem:V'}
The vector space $V'$ has a basis 
  \begin{equation}
    \label{eq:basis'}
    \{x^n:n\ge0\}\cup\{1/d_m(x):m\ge 1\}.        
  \end{equation}
\end{lem}
\begin{proof}
  Let $B$ be the given set. Since the elements in $B$ are linearly independent,
  it suffices to show that $B$ spans $V'$. To do this, it suffices to show that
  $p(x)/d_j(x)\in \mathrm{span}(B)$ for any polynomial $p(x)$ and any integer
  $j\ge1$. 

  By dividing $p(x)$ by $d_j(x)$, we can write
\begin{equation}\label{eq:iterate1}
  \frac{p(x)}{d_j(x)} = q(x) +\frac{r(x)}{d_j(x)},
\end{equation}
where $q(x)$ and $r(x)$ are polynomials and $\deg r(x)<j$. If $r(x)$ is
constant, we have $p(x)/d_j(x)\in\mathrm{span}(B)$. Otherwise, we can write
\begin{equation}\label{eq:iterate2}
  \frac{r(x)}{d_j(x)} = \frac{(a_jx+\lambda_j)r_1(x)+c}{d_j(x)} =
  \frac{r_1(x)}{d_{j-1}(x)}+\frac{c}{d_j(x)},
\end{equation}
where $r_1(x)$ is a polynomial with $\deg r_1(x)<j-1$ and $c$ is a constant. By
iterating \eqref{eq:iterate2} we can express $r(x)/d_j(x)$ as a linear
combination of the elements in $B$. Then by \eqref{eq:iterate1} we have
$p(x)/d_j(x)\in\mathrm{span}(B)$ as desired.
\end{proof}

\begin{lem}\label{lem:span}
  We have $V=V'$.
\end{lem}
\begin{proof}
  By the definitions \eqref{eq:V} and \eqref{eq:V'} of $V$ and $V'$, it is clear
  that $V\subseteq V'$ and $x^n\in V$ for $n\ge0$. By Lemma~\ref{lem:V'}, the
  set in \eqref{eq:basis'} is a basis of $V'$. Thus it suffices to show that
  $1/d_m(x)\in V$ for all $m\ge0$. We prove this by induction on $m$, where the
  base case $m=0$ is clear.

  For $m\ge1$, let $U_{m-1}(x)$ be the quotient of $P_m(x)$ when divided by $a_mx+\lambda_m$:
  \[
    P_m(x) = (a_m x+\lambda_m) U_{m-1}(x) +P_m(-\lambda_m/a_m).
  \]
  Then
  \[
    \frac{P_m(x)}{d_m(x)} = \frac{U_{m-1}(x)}{d_{m-1}(x)} + \frac{P_m(-\lambda_m/a_m)}{d_m(x)},
  \]
and
\begin{equation}
  \label{eq:14}
    \frac{1}{d_m(x)} = \frac{1}{P_m(-\lambda_m/a_m)}
    \left(\frac{P_m(x)}{d_m(x)} - \frac{U_{m-1}(x)}{d_{m-1}(x)}\right).
\end{equation}
Since $U_{m-1}(x)$ is a polynomial of degree $m-1$, by iterating
\eqref{eq:iterate2}, we can write
  \[
\frac{U_{m-1}(x)}{d_{m-1}(x)} = \sum_{i=0}^{m-1} \frac{c_i}{d_i(x)},
  \]
  for some constants $c_0,c_1,\dots,c_{m-1}$. Then by the induction hypothesis,
  $U_{m-1}(x)/d_{m-1}(x)\in V$. Since $P_m(x)/d_m(x)=Q_m(x)\in V$, \eqref{eq:14}
  shows that $1/d_m(x)\in V$, completing the proof.
\end{proof}

Observe that every element of $V$ is of the form $p(x)/d_m(x)$ for some
polynomial $p(x)$ and an integer $m\ge0$. One can find many bases of $V$ as
follows.

\begin{prop}\label{prop:basis}
  Let $\{t_n:n\ge1\}$ be a sequence of nonnegative integers and let
  $\{f_n(x):n\ge1\}$ and $\{g_n(x):n\ge0\}$ be sequences of polynomials
  satisfying the following conditions:
  \begin{itemize}
  \item $\deg (f_n(x))=n+t_n$ for all $n\ge1$, and
  \item $\deg (g_n(x))\le n$ and $g_n(-\lambda_n/a_n)\ne 0$ for all $n\ge0$.
  \end{itemize}
  Then $V$ has a basis given by
  \[
   B= \{f_n(x)/d_{t_n}(x):n\ge1\}\cup\{g_n(x)/d_n(x):n\ge 0\}.    
  \]
\end{prop}
\begin{proof}
  We first show that the elements in $B$ are linearly independent. Suppose that
  \begin{equation}
    \label{eq:16}
  \sum_{i=1}^n C_i \frac{f_i(x)}{d_{t_i}(x)} + \sum_{j=0}^m D_j \frac{g_j(x)}{d_j(x)} = 0,
  \end{equation}
  where $C_i$ and $D_j$ are constants. We must show that the coefficients $C_i$
  and $D_j$ are all zero. Dividing both sides of \eqref{eq:16} by $x^{n}$ and
  taking the limit $x\to\infty$, we obtain
\[
C_n \frac{\mathrm{lead}(f_n(x))}{\mathrm{lead}(d_{t_n}(x))}=0,
\]
where $\mathrm{lead}(p(x))$ is the leading coefficient of $p(x)$. This shows
$C_n=0$. Repeating this process with $x^n$ replaced by $x^i$, for
$i=n-1,n-2,\dots,1$, we obtain that $C_i=0$ for all $1\le i\le n$.
Then \eqref{eq:16} becomes
  \begin{equation}
    \label{eq:16'}
 \sum_{j=0}^m D_j \frac{g_j(x)}{d_j(x)} = 0.
  \end{equation}

  Let $k$ be the number of integers $1\le j\le m$ such that
  $-\lambda_j/a_j=-\lambda_m/a_m$. Multiplying both sides of \eqref{eq:16'} by
  $(a_mx+\lambda_m)^k$ and substituting $x=-\lambda_m/a_m$ we obtain $D_m=0$. In
  this way one can show that $D_j=0$ for all $1\le j\le m$. Then we also have
  $D_0=0$. Therefore the elements in $B$ are linearly independent.

  Now we show that $B$ spans $V$. By Lemmas~\ref{lem:V'} and \ref{lem:span}, it
  suffices to show that $x^n, 1/d_m(x)\in \mathrm{span}(B)$ for all $n\ge1$ and
  $m\ge0$.

  We claim that, for any $m\ge0$ and any polynomial $p(x)$,
\begin{equation}\label{eq:in_span}
\frac{p(x)}{d_m(x)} \in \mathrm{span}(B) \qquad \mbox{if $\deg p(x)\le m$}.
\end{equation}
We proceed by induction on $m$, where the base case $m=0$ is true because
$p(x)/d_0(x)$ is a constant and $g_0(x)/d_0(x)\in B$ is a nonzero constant.
Let $m\ge1$ and suppose \eqref{eq:in_span} holds for all integers less than $m$.
Then by the same argument as in the proof of Lemma~\ref{lem:span} we can write
\[
\frac{p(x)}{d_m(x)} = \sum_{i=0}^{m} \frac{c_i}{d_i(x)}
\]
for some constants $c_i$. By the induction hypothesis, we have $c_i/d_i(x)\in
\mathrm{span}(B)$ for \( 0\le i<m \), and therefore to show $p(x)/d_m(x)\in \mathrm{span}(B)$, it
is enough to show that $1/d_m(x)\in \mathrm{span}(B)$.
Dividing $g_m(x)$ by $(a_mx+\lambda_m)$, we have
\begin{equation}\label{eq:qd}
\frac{g_m(x)}{d_m(x)} = \frac{q(x)}{d_{m-1}(x)} + \frac{g_m(-a_m/\lambda_m)}{d_m(x)},
\end{equation}
where $q(x)$ is a polynomial with $\deg q(x) = \deg g_m(x)-1\le m-1$. By the
induction hypothesis, we have $q(x)/d_{m-1}(x)\in \mathrm{span}(B)$. Then
\eqref{eq:qd} shows that $1/d_m(x)\in \mathrm{span}(B)$ because
$g_m(-a_m/\lambda_m)\ne 0$. Thus \eqref{eq:in_span} is also true for $m$ and the
claim is proved.

By \eqref{eq:in_span}, we have $1/d_m(x)\in \mathrm{span}(B)$ for $m\ge0$.
Therefore it remains to show that $x^n\in \mathrm{span}(B)$ for $n\ge1$.
Dividing $f_n(x)$ by $d_{t_n}(x)$, we have
\begin{equation}
  \label{eq:13}
\frac{f_n(x)}{d_{t_n}(x)} = q_n(x) +  \frac{r_n(x)}{d_{t_n}(x)},
\end{equation}
where $q_n(x)$ and $r_n(x)$ are polynomials with $\deg q_n(x)=n$ and $\deg
r_n(x)<t_n$. By \eqref{eq:in_span}, we have $r_n(x)/d_{t_n}(x)\in
\mathrm{span}(B)$, and therefore \eqref{eq:13} shows $q_n(x)\in
\mathrm{span}(B)$. This implies that
\[
\Span \{x^n:n\ge0\} = \Span (\{1\}\cup \{q_n(x):n\ge1\}) \subseteq \Span(B).
\]
Hence we have $x^n\in \mathrm{span}(B)$ for all $n\ge0$, which completes the
proof.
\end{proof}

As a corollary of Proposition~\ref{prop:basis} we present three notable bases of
$V$. We use the following three choices in Proposition~\ref{prop:basis}:
\begin{align*}
f_n(x)&=x^nd_n(x), & t_n &= n , & g_n(x) &= 1,\\
f_n(x)&=P_n(x)d_n(x), & t_n &= n , & g_n(x) &= P_n(x),\\
f_n(x)&=x^nP_{n+1}(x), & t_n &= n+1 , & g_n(x) &= P_n(x).
\end{align*}

\begin{cor}\label{cor:basis2}
The following are bases of $V$:
\[
    \{x^n:n\ge1\}\cup\{1/d_m(x):m\ge 0\},
\]
\[
    \{P_n(x):n\ge1\}\cup\{Q_m(x):m\ge0\},
\]
\[
\{x^{n}Q_{n+1}(x):n\ge1\}\cup \{Q_m(x):m\ge0\}.
\]
\end{cor}

Now we give a detailed proof of Theorem~\ref{thm:unique L}. Our proof is
basically the same as the proof in \cite{IsmailMasson} with some details provided.

\begin{proof}[Proof of Theorem~\ref{thm:unique L}]
  By Corollary~\ref{cor:basis2},
\[
B:= \{x^{m-1}Q_{m}(x):m\ge2\}\cup \{Q_m(x):m\ge0\}
\]
is a basis of $V$. Since the values of $\LL$ on $\{1\}\cup\{x^n Q_m(x):0\le
n<m\}$ are given, and this set contains $B$, if $\LL$ exists, it is unique. For
the existence, we construct $\LL$ by defining its values on the basis $B$ as
follows:
\[
\LL(1)=1,\qquad \LL(Q_m(x))=\LL(x^{m-1}Q_m(x))=0,\qquad \mbox{for $m\ge1$} .
\]
It suffices to show that $\LL$ satisfies 
\begin{equation}\label{eq:LL=0}
\LL(x^n Q_m(x)) = 0  \qquad \mbox{if $0\le n< m$}.
\end{equation}

We prove \eqref{eq:LL=0} by induction on $(n,m)$. The base case $n=0$ is true by
definition of $\LL$. Let $1\le n<m$ and assume that \eqref{eq:LL=0} is true for all
pairs $(n',m')\ne (n,m)$ such that $0\le n'<m'$, $n'\le n$ and $m'\le m$. We now
show that it is also true for $(n,m)$. Since it is true for $(m-1,m)$ by
definition of $\LL$, we may assume $1\le n\le m-2$. By \eqref{eq:favard2} we
have
\[
(a_mx+\lambda_m)Q_m(x) = (x-b_{m-1})Q_{m-1}(x) - Q_{m-2}(x).
\]
Multiplying both sides of the above equation by $x^{n-1}$ we have
\[
a_mx^n Q_m(x)+\lambda_mx^{n-1} Q_m(x) = 
 x^nQ_{m-1}(x) -b_{m-1}x^{n-1}Q_{m-1}(x) - x^{n-1} Q_{m-2}(x).
\]
Since $0\le n-1\le m-3$, by the induction hypothesis,
taking $\LL$ on both sides gives
\[
a_m \LL(x^nQ_m(x))=0.
\]
Since $a_m\ne 0$, we obtain $\LL(x^nQ_m(x))=0$. Hence \eqref{eq:LL=0} is also
true for $(n,m)$ and the proof is completed by induction.
\end{proof}

\section{Moments of type $R_I$ orthogonal polynomials}
\label{sec:moments-ops-type}

Recall from Theorem~\ref{thm:unique L} that there is a unique linear functional
$\LL$ on $V$ satisfying $\LL(1)=1$ and
\[
\LL(x^n Q_m(x)) = 0 \qquad \mbox{if $0\le n< m$}.
\]
In this section we give combinatorial interpretations for 
\begin{align*}
\mu_n&:=\LL(x^n),\\
  \mu_{n,m}&:=\LL(x^n Q_m(x)),\\
  \mu_{n,m,\ell}&:=\LL(x^n P_m(x) Q_\ell(x)),\\
\rho_{n,m,\ell}&:=\LL(x^n P_m(x) P_\ell(x))
\end{align*}
 in terms of lattice paths called Motzkin-Schr\"oder
paths. We also give a recursive formula for 
\[
\nu_{n,m} = \LL\left(x^n/d_m(x)\right).
\]
 Note that $Q_n(x)Q_m(x)\not\in V$ in general, so we
do not consider $\LL(Q_n(x)Q_m(x))$.

\subsection{Combinatorial interpretations for $\mu_n$ and 
$\mu_{n,m}$}
\aftersubsection

In this subsection we give combinatorial interpretations
for $\mu_n$ and $\mu_{n,m}$ using recursive formulas.

Note that $\mu_{n,m}$ has the initial conditions given by
\begin{equation}
  \label{eq:init}
\mu_{0,0}=1, \qquad \mu_{n,m}=0\quad \mbox{for $0\le n<m$}. 
\end{equation}
The following lemma
gives a recurrence, which determines $\mu_{n,m}$ since the recurrence decreases
either $n$ or $n-m$

\begin{lem}\label{lem:mu_rec}
  For $n,m\ge0$ with $(n,m)\ne (0,0)$, we have
\[
   \mu_{n,m} = 
    \begin{cases}
      0, & \mbox{if $n<m$},\\
      a_{m+1}\mu_{n,m+1} + b_{m}\mu_{n-1,m} +\mu_{n-1,m-1}+
    \lambda_{m+1} \mu_{n-1,m+1}, & \mbox{if $n\ge m$},
    \end{cases}
\]  
where $\mu_{i,j}=0$ if $j<0$. 
\end{lem}
\begin{proof}
  We have already seen that $\mu_{n,m}=0$ for $n< m$. Suppose
  $n\ge m$. Since $(n,m)\ne (0,0)$ we have $n\ge1$. Then
  \begin{align*}
    a_{m+1} \mu_{n,m+1} &= \LL\left(\frac{a_{m+1}  x^n P_{m+1}(x)}{d_{m+1}(x)}\right) \\
    & = \LL\left(\frac{x^{n-1}(a_{m+1}x+\lambda_{m+1}) P_{m+1}(x)}{d_{m+1}(x)}-\frac{\lambda_{m+1}x^{n-1}P_{m+1}(x)}{d_{m+1}(x)}\right)\\
    &= \LL\left(\frac{x^{n-1}((x-b_m)P_m(x) - (a_mx+\lambda_m)P_{m-1}(x))}{d_m(x)}\right) -
      \lambda_{m+1}\mu_{n-1,m+1}\\
                &= \mu_{n,m} - b_m\mu_{n-1,m}-\mu_{n-1,m-1} -\lambda_{m+1}\mu_{n-1,m+1}.
  \end{align*}
  Thus
  \[
    \mu_{n,m} = a_{m+1}\mu_{n,m+1} + b_{m}\mu_{n-1,m} +\mu_{n-1,m-1}+
    \lambda_{m+1} \mu_{n-1,m+1},
  \]
as desired.
\end{proof}

\begin{defn}
A \emph{Motzkin path} is a path on or above the $x$-axis consisting of up steps
$U=(1,1)$, horizontal steps $H=(1,0)$, and down steps $D=(1,-1)$. A \emph{Schr\"oder
  path} is a path on or above the $x$-axis consisting of up steps $U=(1,1)$,
horizontal steps $H=(1,0)$, and vertical down steps $V=(0,-1)$.
\end{defn}

\begin{remark}\label{rem:sch}
  A Schr\"oder path is more commonly defined as a path consisting of up steps
  $(1,1)$, double horizontal steps $(2,0)$, and down steps $(1,-1)$. One can
  easily convert from our definition of a Schr\"oder path to this one by
  changing each horizontal step $(1,0)$ to a double horizontal step $(2,0)$, and
  each vertical down step $(0,-1)$ to a diagonal down step $(1,-1)$.
\end{remark}

Now we define another lattice path, which contains both Motzkin and Schr\"oder
paths.

\begin{defn}
\label{defn:MSpath}
A \emph{Motzkin-Schr\"oder path} is a path on or above the $x$-axis consisting
of up steps $U=(1,1)$, horizontal steps $H=(1,0)$, vertical down step
$V=(0,-1)$, and diagonal down steps $D=(1,-1)$. For a Motzkin-Schr\"oder path
$\pi$, the \emph{weight} $\wt(\pi)$ of $\pi$ is the product of the weight of
each step, where every up step has weight $1$, a horizontal step starting at
height $k$ has weight $b_{k}$, a vertical down step $(0,-1)$ starting at height
$k$ has weight $a_{k}$, and a diagonal down step $(1,-1)$ starting at height $k$
has weight $\lambda_k$. See Figure~\ref{fig:MS}.
\end{defn} 

\begin{remark}
  In \cite{Kim1992}, Kim found a combinatorial interpretation for
  moments of biorthogonal polynomials using lattice paths. The steps are up
  steps, horizontal steps, and $d$ types of down steps. If $d=2$, these lattice
  paths are equivalent to our Motzkin-Schr\"oder paths.
\end{remark}

Let $\MS((a,b)\to(c,d))$ denote the set of Motzkin-Schr\"oder paths from
$(a,b)$ to $(c,d)$ and let
\[
\MS_{n,m} := \MS((0,0)\to(n,m)), \qquad \MS_{n} := \MS_{n,0}.
\]
We also denote by $\Mot((a,b)\to(c,d))$ (resp.~$\Sch((a,b)\to(c,d))$) the set of
Motzkin (resp.~Schr\"oder) paths from $(a,b)$ to $(c,d)$. The sets $\Mot_{n,m}$,
$\Mot_{n}$, $\Sch_{n,m}$, and $\Sch_{n}$ are defined similarly.

\begin{figure}
  \centering
\begin{tikzpicture}[scale=0.75]
    \draw[help lines] (0,0) grid (13,3);
    \foreach \y in {0,...,3} \draw node at (-.5,\y) {$\y$};
    \draw[line width = 1.5pt] (0,3)-- ++(0,-1)-- ++(0,-1)-- ++(1,-1)-- ++(1,1)-- ++(1,1)-- ++(0,-1)-- ++(1,-1)-- ++(1,0)-- ++(1,1)-- ++(1,1)-- ++(1,1)-- ++(1,-1)-- ++(1,0)-- ++(0,-1)-- ++(1,1)-- ++(1,0)-- ++(0,-1)-- ++(1,0);
    \node at (0.4,2.5) {$a_{3}$};
    \node at (0.4,1.5) {$a_{2}$};
    \node at (0.8,0.8) {$\lambda_{1}$};
    \node at (3.4,1.5) {$a_{2}$};
    \node at (3.8,0.8) {$\lambda_{1}$};
    \node at (4.5,0.3) {$b_{0}$};
    \node at (8.8,2.8) {$\lambda_{3}$};
    \node at (9.5,2.3) {$b_{2}$};
    \node at (10.4,1.8) {$a_{2}$};
    \node at (11.5,2.3) {$b_{2}$};
    \node at (12.4,1.8) {$a_{2}$};
    \node at (12.7,1.3) {$b_{1}$};
\end{tikzpicture}
\caption{A Motzkin-Schr\"oder path $\pi$ from $(0,3)$ to $(13,1)$ with
  $\wt(\pi)=a_2^4a_3b_0b_1b_2^2\lambda_1^2\lambda_3$.}
  \label{fig:MS}
\end{figure}
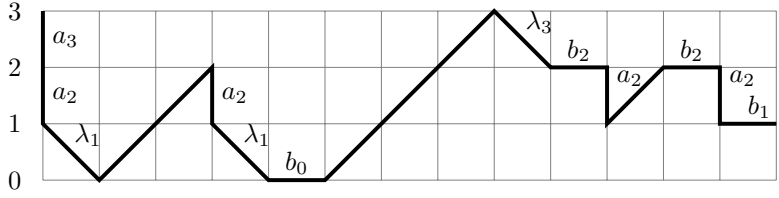

\begin{thm}\label{thm:R1mu}
  For $n,m\ge0$, we have
  \[
    \mu_{n,m} = \LL(x^nQ_m(x)) = \sum_{\pi\in\MS_{n,m}} \wt(\pi).
  \]
\end{thm}
\begin{proof}
  Let $\mu'_{n,m}$ be the right hand side. Then one can easily check that
  $\mu'_{n,m}$ satisfies the same initial conditions in \eqref{eq:init} and the
  recurrence relation in Lemma~\ref{lem:mu_rec} as $\mu_{n,m}$. Therefore
  $\mu_{n,m}=\mu'_{n,m}$ for all $n,m\ge0$.
\end{proof}

\begin{cor}\label{cor:R1cfrac}
  We have
  \[
    \mu_{n} = \sum_{\pi\in\MS_{n}} \wt(\pi).
  \]
Equivalently,
  \[
    \sum_{n\ge0} \mu_{n} x^n =
    \cfrac{1}{
      1-b_0x- \cfrac{a_1x+\lambda_1x^2}{
        1-b_1x- \cfrac{a_2x+\lambda_2x^2}{
          \ddots}}}.
  \]
\end{cor}
\begin{proof}
  The first statement is the special case $m=0$ of Theorem~\ref{thm:R1mu}. The
  second statement follows from the first using Flajolet's theory
  \cite{Flajolet1980}.
\end{proof}

For example, by Corollary~\ref{cor:R1cfrac}, the moments $\mu_n$ for $n=0,1,2$
can be computed as
\begin{align*}
  \mu_{0} &= 1,\\
  \mu_{1} &= b_0+a_1,\\
  \mu_{2} &= b_0^2 + \lambda_1 + 2a_1b_0 + a_2a_1+b_1a_1+a_1^2.
\end{align*}

\begin{remark}
  Lattice paths containing Motzkin-Schr\"oder paths are considered in
  \cite{Dziemianczuk}. The number of Motzkin-Schr\"oder paths is listed in
  \cite[A064641]{OEIS}. If $b_n=a_n=0$ and $\lambda_n=1$, then $\mu_{2n+1}=0$
  and $\mu_{2n}=C_n$, the $n$th Catalan number. If $b_n=\lambda_n=0$ and
  $a_n=1$, then $\mu_{n}=C_n$.
\end{remark}

By taking $a_n=b_n=\lambda_n=1$ for all $n$ in Corollary~\ref{cor:R1cfrac} we
obtain the following result, which also appears in \cite{Dziemianczuk}.

\begin{prop}
\label{prop: CF}
The generating function for the number of Motzkin-Schr\"oder paths is given by
  \[
    \sum_{n\ge0} |\MS_n| x^n = \frac{1-x-\sqrt{1-6x-3x^2}}{2(x+x^2)}
    =1 + 2x+7x^2+29x^3+133x^4+ 650x^5+\cdots.
   \]
\end{prop}

\subsection{Combinatorial interpretations for $\mu_{n,m,\ell}$ and
  $\rho_{n,m,\ell}$}
\aftersubsection

In this subsection we give a combinatorial interpretation for $P_n(x)$ in terms
of tilings (Theorem~\ref{thm:favard}). This allows us to write
$\mu_{n,m,\ell}=\LL(x^n P_m(x) Q_\ell(x))$ as a signed sum of $\mu_{n,m}=\LL(x^n
P_m(x))$. We will show that these are positive sums by finding a sign-reversing
involution which cancel all negative terms. Using a similar argument we will
also find a positive formula for $\rho_{n,m,\ell}=\LL(x^n P_m(x) P_\ell(x))$.

\begin{defn}
A \emph{(bicolored) Favard tiling of size $n$} is a tiling of a $1\times n$
square board with tiles where each tile is a domino or a monomino and is colored
black or red. We label the squares in the $1\times n$ board by $1,2,\dots,n$
from left to right. The set of Favard tilings of size $n$ is denoted by $\FT_n$.
We define \( \FT_0 \) to be the set consisting of the empty tiling.
\end{defn}

For $T\in\FT_n$, we define $\bm(T)$ (resp.~$\bd(T)$, $\rm(T)$, and $\rd(T)$) to
be the number of black monominos (resp.~black dominos, red monominos, and red
dominos) in $T$. We also define
\[
  \wt(T) = \prod_{\tau\in T} \wt(\tau),
\]
where 
\[
  \wt(\tau) =
  \begin{cases}
    1 & \mbox{if $\tau$ is a black monomino,}\\
   -b_{i-1} & \mbox{if $\tau$ is a red monomino with (largest) entry $i$,}\\
   -a_{i-1} & \mbox{if $\tau$ is a black domino with largest entry $i$,}\\
     -\lambda_{i-1}  & \mbox{if $\tau$ is a red domino with largest entry $i$.}
  \end{cases}
\]
For example, see Figure~\ref{fig:tiling}.

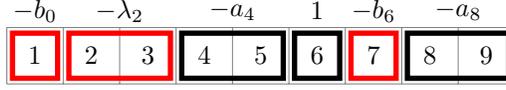
\begin{figure}
  \centering
\begin{tikzpicture}[scale=0.75]
  \draw [help lines] (0,0) grid (9,1);
  \foreach \x in {1,...,9} \draw node at (\x-0.5,0.5) {\x};
  \LRM0 \LRD2 \LBD4 \LBM5 \LRM6 \LBD8
\end{tikzpicture}
\caption{A Favard tiling $T\in\FT_9$ with $\wt(T)=-a_4a_8b_0b_6\lambda_2$. This tiling contributes
$\wt(T) x^{\bm(T)+\bd(T)} = -a_4a_8b_0b_6\lambda_2x^3$ to $P_8(x)$.}
  \label{fig:tiling}
\end{figure}

The recurrence \eqref{eq:favard1} gives the following combinatorial interpretation for $P_n(x)$.

\begin{thm}\label{thm:favard}
For $n\ge0$, we have
  \[
    P_n(x) = \sum_{T\in\FT_n} \wt(T) x^{\bm(T)+\bd(T)}.
  \]
\end{thm}
\begin{proof}
  Let $U_n(x)$ denote the right hand side. Then by definition one can easily
  check that, for $n\ge0$, we have
\[
    U_{n+1}(x) = (x-b_n)U_n(x) -(a_n x+\lambda_n )U_{n-1}(x)
\]
where $U_{-1}(x)=0$ and $U_{0}(x)=1$. Therefore, by \eqref{eq:favard1}, $U_n(x)$
and $P_n(x)$ satisfy the same recurrence relation and the same initial
conditions. This shows the theorem.
\end{proof}

Now we give a combinatorial interpretation for $\mu_{n,m,\ell}$.

\begin{thm}\label{thm:mu_nml}
For $n,m,\ell\ge0$, we have
  \[
   \mu_{n,m,\ell} = \LL(x^n P_m(x) Q_\ell(x)) =
    \sum_{\pi\in \MS((0,m)\to (n,\ell))} \wt(\pi).
  \]
\end{thm}
\begin{proof}
 We will find a sign-reversing involution on a larger set whose fixed point set
 is given by the Motzkin-Schr\"oder paths in this theorem.
  
Applying Theorem~\ref{thm:favard} to $P_m(x)$ and using Theorem~\ref{thm:R1mu}, we have
\begin{equation}
  \label{eq:18}
    \LL(x^n P_m(x) Q_\ell(x)) = \sum_{T\in\FT_m} \wt(T) \LL(x^{n+\bm(T)+\bd(T)}Q_\ell(x))
    =\sum_{(\pi,T)\in X} \wt(\pi)\wt(T),
\end{equation}
where $X$ is the set of pairs $(\pi,T)$ of a Motzkin-Schr\"oder path $\pi\in
\MS_{t,\ell}$ and a Favard tiling $T\in\FT_m$ satisfying $t=n+\bm(T)+\bd(T)$.
The sign-reversing involution on $X$ will remove or add a horizontal step or a
peak ($(U,V)$ or $(U,D)$) in $\pi$, and modify $T$ accordingly.

  Consider $(\pi,T)\in X$ and write $\pi=S_1\dots S_r$ as a sequence of steps.
  Suppose that $i$ and $j$ are the largest integers such that $\pi$ starts with
  $i$ up steps and $T$ starts with $j$ black monominos.
  \begin{description}
  \item[Case 1] $j\ge i+1$. In this case we have $i+1\le \bm(T)\le
    n+\bm(T)+\bd(T)=t$. Therefore $\pi$ must have the $(i+1)$st step. We define
    $\pi'$ and $T'$ in the following three cases depending on the step
    $S_{i+1}$.
\begin{description}
\item[Case 1-a] $S_{i+1}$ is a horizontal step. In this case let
    \[
      \pi' = S_1\dots \widehat{S}_{i+1} \dots S_r,
\]
and define $T'$ to be the Favard tiling obtained from $T$ by replacing the black
monomino at position $i+1$ by a red monomino. Here the notation
$\widehat{S}_{i+1}$ means that $S_{i+1}$ is removed from the sequence. See Figure~\ref{fig:case-a}.
\item[Case 1-b] $S_{i+1}$ is a vertical down step. In this case let
    \[
      \pi' = S_1\dots \widehat{S}_{i}\widehat{S}_{i+1} \dots S_r,
\]
and define $T'$ to be the Favard tiling obtained from $T$ by replacing the two
black monominos at positions $i$ and $i+1$ by a black domino. See Figure~\ref{fig:case-b}.
\item[Case 1-c] $S_{i+1}$ is a diagonal down step. In this case let
    \[
      \pi' = S_1\dots \widehat{S}_{i}\widehat{S}_{i+1} \dots S_r,
\]
and define $T'$ to be the Favard tiling obtained from $T$ by replacing the two
black monominos at positions $i$ and $i+1$ by a red domino. See Figure~\ref{fig:case-c}.
\end{description}

\item[Case 2] $j\le i$ and $j<m$. In this case $T$ contains a tile, say $A$,
  with entry $j+1$. We define $\pi'$ and $T'$ in the following three cases
  depending on the tile $A$.
  \begin{description}
  \item[Case 2-a] $A$ is a red monomino. In this case let
    \[
      \pi' = S_1\dots S_j H S_{j+1} \dots S_r,
\]
and define $T'$ to be the Favard tiling obtained from $T$ by replacing $A$ by a
black monomino. See Figure~\ref{fig:case-a}.
\item[Case 2-b] $A$ is a black domino. In this case let
    \[
      \pi' = S_1\dots S_j UV S_{j+1} \dots S_r,
\]
and define $T'$ to be the Favard tiling obtained from $T$ by replacing $A$ by two
black monominos. See Figure~\ref{fig:case-b}.
\item[Case 2-c] $A$ is a red domino. In this case let
    \[
      \pi' = S_1\dots S_j UD S_{j+1} \dots S_r,
\]
and define $T'$ to be the Favard tiling obtained from $T$ by replacing $A$ by two
black monominos. See Figure~\ref{fig:case-c}.
\end{description}
\item[Case 3] $j\le i$ and $j=m$. In this case define $\pi'=\pi$ and $T'=T$.
  See Figure~\ref{fig:case3}.
  \end{description}

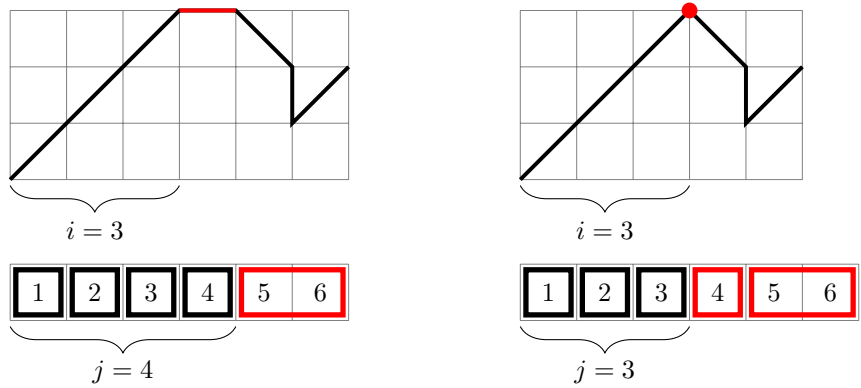
\begin{figure}
  \centering
\begin{tikzpicture}[scale=0.75]
    \draw[help lines] (0,0) grid (6,3);
   \draw[line width = 1.5pt] (0,0)--++(3,3)--++(1,0)--++(1,-1)--++(0,-1)--++(1,1);
   \draw[line width = 1.5pt, red] (3,3)--++(1,0);
  \draw [decorate,decoration={brace,amplitude=10pt},xshift=-0pt,yshift=-3pt]
(3,0) -- (0,0) node [black,midway,yshift=-0.6cm] {$i=3$};
\begin{scope}[shift={(0,-2.5)}]
  \draw [help lines] (0,0) grid (6,1);
  \foreach \x in {1,...,6} \draw node at (\x-0.5,0.5) {\x};
  \BM0 \BM1 \BM2 \BM3 \RD5
  \draw [decorate,decoration={brace,amplitude=10pt},xshift=-0pt,yshift=-3pt]
(4,0) -- (0,0) node [black,midway,yshift=-0.6cm] {$j=4$};
\end{scope}
\end{tikzpicture}   \qquad\qquad\qquad
\begin{tikzpicture}[scale=0.75]
    \draw[help lines] (0,0) grid (5,3);
   \draw[line width = 1.5pt] (0,0)--++(3,3)--++(1,-1)--++(0,-1)--++(1,1);
   \remove(3,3)
  \draw [decorate,decoration={brace,amplitude=10pt},xshift=-0pt,yshift=-3pt]
(3,0) -- (0,0) node [black,midway,yshift=-0.6cm] {$i=3$};
\begin{scope}[shift={(0,-2.5)}]
  \draw [help lines] (0,0) grid (6,1);
  \foreach \x in {1,...,6} \draw node at (\x-0.5,0.5) {\x};
  \BM0 \BM1 \BM2 \RM3 \RD5
  \draw [decorate,decoration={brace,amplitude=10pt},xshift=-0pt,yshift=-3pt]
(3,0) -- (0,0) node [black,midway,yshift=-0.6cm] {$j=3$};
\end{scope}
\end{tikzpicture}   
\caption{A pair $(\pi,T)\in X$ in Case 1-a on the left and the corresponding
  pair $(\pi',T')$ in Case 2-a on the right, for $(n,m,\ell)=(2,6,2)$.
The horizontal step starting at $(3,3)$ in $\pi$ is collapsed to a point.}
  \label{fig:case-a}
\end{figure}

\begin{figure}
  \centering
\begin{tikzpicture}[scale=0.75]
    \draw[help lines] (0,0) grid (6,3);
   \draw[line width = 1.5pt] (0,0)--++(3,3)--++(0,-2)--++(1,1)--++(1,0)--++(1,1)--++(0,-1);
   \draw[line width = 1.5pt, red] (2,2)--++(1,1)--++(0,-1);
  \draw [decorate,decoration={brace,amplitude=10pt},xshift=-0pt,yshift=-3pt]
(3,0) -- (0,0) node [black,midway,yshift=-0.6cm] {$i=3$};
\begin{scope}[shift={(0,-2.5)}]
  \draw [help lines] (0,0) grid (6,1);
  \foreach \x in {1,...,6} \draw node at (\x-0.5,0.5) {\x};
  \BM0 \BM1 \BM2 \BM3 \RD5
  \draw [decorate,decoration={brace,amplitude=10pt},xshift=-0pt,yshift=-3pt]
(4,0) -- (0,0) node [black,midway,yshift=-0.6cm] {$j=4$};
\end{scope}
\end{tikzpicture}   \qquad\qquad\qquad
\begin{tikzpicture}[scale=0.75]
    \draw[help lines] (0,0) grid (5,3);
   \draw[line width = 1.5pt] (0,0)--++(2,2)--++(0,-1)--++(1,1)--++(1,0)--++(1,1)--++(0,-1);
   \remove(2,2)
  \draw [decorate,decoration={brace,amplitude=10pt},xshift=-0pt,yshift=-3pt]
(2,0) -- (0,0) node [black,midway,yshift=-0.6cm] {$i=2$};
\begin{scope}[shift={(0,-2.5)}]
  \draw [help lines] (0,0) grid (6,1);
  \foreach \x in {1,...,6} \draw node at (\x-0.5,0.5) {\x};
  \BM0 \BM1 \BD3 \RD5
  \draw [decorate,decoration={brace,amplitude=10pt},xshift=-0pt,yshift=-3pt]
(2,0) -- (0,0) node [black,midway,yshift=-0.6cm] {$j=2$};
\end{scope}
\end{tikzpicture}   
\caption{A pair $(\pi,T)\in X$ in Case 1-b on the left and the corresponding
  pair $(\pi',T')$ in Case 2-b on the right, for $(n,m,\ell)=(2,6,2)$.
The peak $(U,V)$ starting at $(2,2)$ in $\pi$ is collapsed to a point.}
  \label{fig:case-b}
\end{figure}
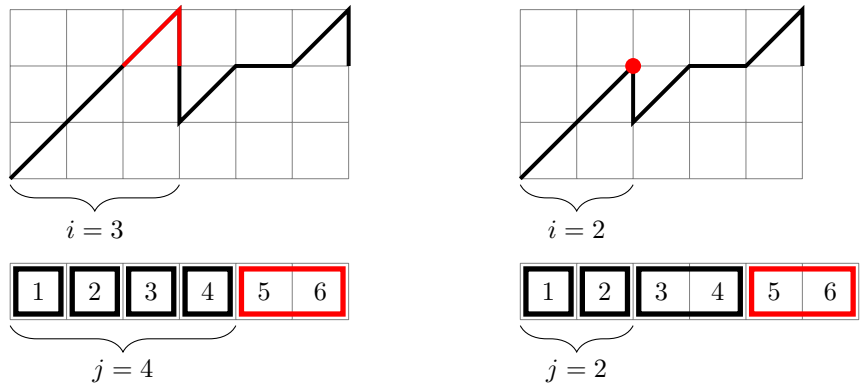

\begin{figure}
  \centering
\begin{tikzpicture}[scale=0.75]
    \draw[help lines] (0,0) grid (6,3);
   \draw[line width = 1.5pt] (0,0)--++(3,3)--++(1,-1)--++(1,0)--++(0,-1)--++(1,1);
   \draw[line width = 1.5pt, red] (2,2)--++(1,1)--++(1,-1);
  \draw [decorate,decoration={brace,amplitude=10pt},xshift=-0pt,yshift=-3pt]
(3,0) -- (0,0) node [black,midway,yshift=-0.6cm] {$i=3$};
\begin{scope}[shift={(0,-2.5)}]
  \draw [help lines] (0,0) grid (6,1);
  \foreach \x in {1,...,6} \draw node at (\x-0.5,0.5) {\x};
  \BM0 \BM1 \BM2 \BM3 \RD5
  \draw [decorate,decoration={brace,amplitude=10pt},xshift=-0pt,yshift=-3pt]
(4,0) -- (0,0) node [black,midway,yshift=-0.6cm] {$j=4$};
\end{scope}
\end{tikzpicture}   \qquad\qquad\qquad
\begin{tikzpicture}[scale=0.75]
    \draw[help lines] (0,0) grid (4,3);
   \draw[line width = 1.5pt] (0,0)--++(2,2)--++(1,0)--++(0,-1)--++(1,1);
   \remove(2,2)
  \draw [decorate,decoration={brace,amplitude=10pt},xshift=-0pt,yshift=-3pt]
(2,0) -- (0,0) node [black,midway,yshift=-0.6cm] {$i=2$};
\begin{scope}[shift={(0,-2.5)}]
  \draw [help lines] (0,0) grid (6,1);
  \foreach \x in {1,...,6} \draw node at (\x-0.5,0.5) {\x};
  \BM0 \BM1 \RD3 \RD5
  \draw [decorate,decoration={brace,amplitude=10pt},xshift=-0pt,yshift=-3pt]
(2,0) -- (0,0) node [black,midway,yshift=-0.6cm] {$j=2$};
\end{scope}
\end{tikzpicture}   
\caption{A pair $(\pi,T)\in X$ in Case 1-c on the left and the corresponding
  pair $(\pi',T')$ in Case 2-c on the right, for $(n,m,\ell)=(2,6,2)$.
The peak $(U,D)$ starting at $(2,2)$ in $\pi$ is collapsed to a point.}
  \label{fig:case-c}
\end{figure}
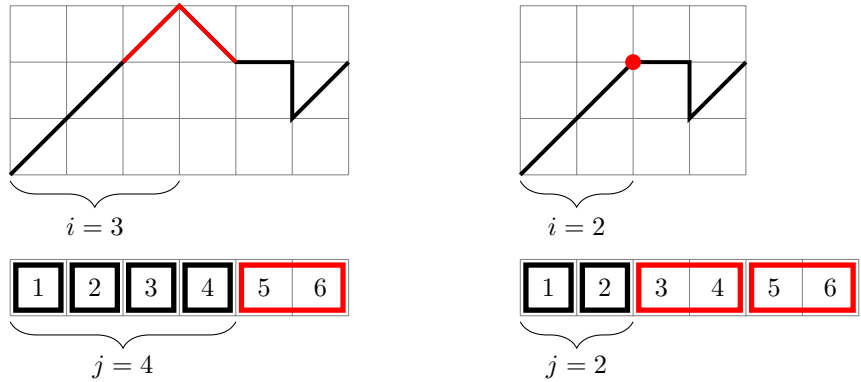

\begin{figure}
  \centering
  \begin{tikzpicture}[scale=0.75]
    \draw[help lines] (0,0) grid (8,6);
   \draw[line width = 1.5pt] (0,0)--++(6,6)--++(1,0)--++(0,-3)--++(1,1)--++(0,-2);
  \draw [decorate,decoration={brace,amplitude=10pt},xshift=-0pt,yshift=-3pt]
(6,0) -- (0,0) node [black,midway,yshift=-0.6cm] {$i=6$};
\begin{scope}[shift={(0,-2.5)}]
  \draw [help lines] (0,0) grid (6,1);
  \foreach \x in {1,...,6} \draw node at (\x-0.5,0.5) {\x};
  \BM0 \BM1 \BM2 \BM3 \BM4 \BM5
  \draw [decorate,decoration={brace,amplitude=10pt},xshift=-0pt,yshift=-3pt]
(6,0) -- (0,0) node [black,midway,yshift=-0.6cm] {$j=6$};
\end{scope}
\end{tikzpicture}   
\caption{A pair $(\pi,T)\in X$ in Case 3 for $(n,m,\ell)=(2,6,2)$. In this case
  $(\pi,T) = (\pi',T')$ is a fixed point.}
  \label{fig:case3}
\end{figure}
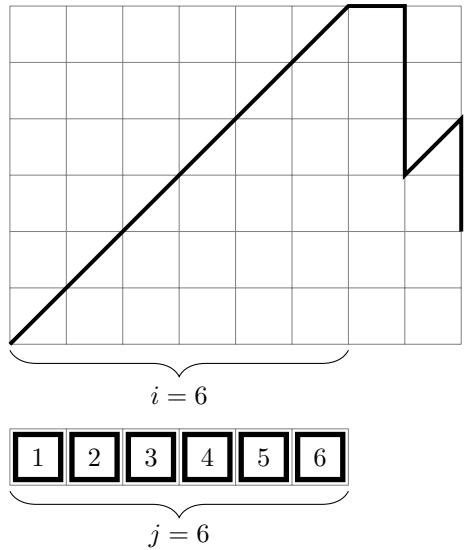

It is straightforward to verify that $(\pi,T)\mapsto(\pi',T')$ is a
sign-reversing involution on $X$ whose fixed points are the pairs $(\pi,T)$ with
$\pi\in \MS((0,0)\to(m+n,\ell))$ and $T\in \FT_m$ such that the first $m$ steps
of $\pi$ are up steps and $T$ consists of $m$ black monominos. Note that if
$(\pi,T)$ is a fixed point, then $\wt(T)=1$ and $\wt(\pi)=\wt(\pi_{\ge m})$,
where $\pi_{\ge m}$ is the subpath of $\pi$ from $(m,m)$ to $(m+n,\ell)$. This
shows that
\begin{equation}
  \label{eq:19}
  \sum_{(\pi,T)\in X} \wt(\pi)\wt(T)= \sum_{\pi\in \MS((m,m)\to(m+n,\ell))} \wt(\pi)
  = \sum_{\pi\in \MS((0,m)\to(n,\ell))} \wt(\pi).
\end{equation}
Then the theorem follows from \eqref{eq:18} and \eqref{eq:19}.
\end{proof}

Now we list a number of special cases of Theorem~\ref{thm:mu_nml}.

First of all, if $a_n= 0$, we obtain Viennot's result.
\begin{cor}\cite[Proposition~17 on page I-15]{ViennotLN}
If \( a_i=0 \) for all \( i\ge1 \), then we have
\[
\LL(x^n P_m(x)P_\ell(x)) = \lambda_1\dots\lambda_\ell \sum_{\pi\in \Mot((0,m)\to (n,\ell))} \wt(\pi).
\]
\end{cor}

In the next section we will show that if $\lambda_n=0$ and $\ell=0$ in
Theorem~\ref{thm:mu_nml}, then we obtain Kamioka's result
\cite[Lemma~3.1]{Kamioka2014} on Laurent biorthogonal polynomials.

If $m=0$ or $\ell=0$ in Theorem~\ref{thm:mu_nml}, we obtain the following
corollary.

\begin{cor}\label{cor:L(x^nP_m)}
For $n,m\ge0$, we have
  \begin{align*}
    \LL(x^n P_m(x)) &= \sum_{\pi\in \MS((0,m)\to (n,0))} \wt(\pi),\\
    \LL(x^n Q_m(x)) &= \sum_{\pi\in \MS((0,0)\to (n,m))} \wt(\pi).
  \end{align*}
\end{cor}

If $n=0$ in Theorem~\ref{thm:mu_nml}, we obtain the following corollary.

\begin{cor}\label{cor:L(P_mQ_ell)}
  We have
  \[
    \LL(P_n(x)Q_m(x)) = 
    \begin{cases}
     0 & \mbox{if $n<m$},\\
     a_{m+1}a_{m+2}\dots a_n & \mbox{if $n\ge m$}.
    \end{cases}
  \]
  In particular,
  \begin{align*}
    \LL(P_n(x)Q_m(x)) &= \delta_{n,m}, \qquad \mbox{if $0\le n\le m$},\\
    \LL(P_n(x)) &= a_1 a_2 \cdots a_n,\\
    \LL(P_n(x)(Q_m(x)-a_{m+1}Q_{m+1}(x))) &= \delta_{n,m}.
  \end{align*}
\end{cor}

If $m=0$ and $n=\ell$, or $n=0$ and $m=\ell$, we obtain the following corollary, which is equivalent
to \cite[Corollary 2.2]{IsmailMasson}.

\begin{cor}\label{cor:L(x^nQ_n)=1}
  We have
  \[
    \LL(x^nQ_n(x)) = \LL(P_n(x)Q_n(x)) = 1.
  \]
\end{cor}

Using Corollary~\ref{cor:L(P_mQ_ell)} we can find the coefficients in the
expansion of an arbitrary polynomial as a linear combination of $P_n(x)$.

\begin{prop} Let $p(x)$ be a polynomial in $x$, and expand
$$
p(x)=\sum_{m=0}^\infty c_m P_m(x).
$$
Then
$$
c_m= \LL\left(p(x) (Q_m(x)-a_{m+1}Q_{m+1}(x))\right).
$$
\begin{proof}
This follows immediately from the last identity in Corollary~\ref{cor:L(P_mQ_ell)}.
\end{proof}
\end{prop}

The following theorem implies that $\rho_{n,m,\ell}=\LL(x^n P_m(x) P_\ell(x))$ is a
positive polynomial in $a_k,b_k$, and $\lambda_k$.

\begin{thm}\label{thm:x^nP_mP_l}
  For $n,m,\ell\ge0$, we have
  \[
    \rho_{n,m,\ell} = \LL(x^n P_m(x) P_\ell(x)) = \sum_{\pi} \wt(\pi),
  \]
  where the sum is over all Motzkin-Schr\"oder paths from $(0,m)$ to
  $(n+\ell,0)$ such that the last $\ell$ steps consist only of vertical down steps
  and diagonal down steps.
\end{thm}
\begin{proof}
  By Theorem~\ref{thm:favard} and Theorem~\ref{thm:mu_nml}, we have
  \[
    \LL(x^n P_m(x) P_\ell(x)) = \sum_{T\in\FT_\ell} \wt(T) \LL(x^{n+\bm(T)+\bd(T)}P_m(x))
    =\sum_{(\pi,T)\in X} \wt(\pi)\wt(T),
  \]
  where $X$ is the set of pairs $(\pi,T)$ of a Motzkin-Schr\"oder path $\pi$
  from $(0,m)$ to $(n+t,0)$ and a Favard tiling $T\in\FT_\ell$ satisfying
  $t=n+\bm(T)+\bd(T)$. 

  By the same argument as in the proof of Theorem~\ref{thm:mu_nml}, we can find
  a sign-reversing involution on $X$ whose fixed points are exactly the
  Motzkin-Schr\"oder paths described in this theorem. The only difference in the
  construction of the sign-reversing involution is that we write
  $\pi=S_rS_{r-1}\dots S_1$ and let $i$ be the largest integer such that the
  last $i$ steps $S_1,\dots,S_i$ consist of vertical and diagonal down steps. We omit the
  details.
\end{proof}

If $\ell=0$ in Theorem~\ref{thm:x^nP_mP_l}, we obtain
Corollary~\ref{cor:L(x^nP_m)}. If $n=0$ in Theorem~\ref{thm:x^nP_mP_l}, we
obtain the following corollary.

\begin{cor}
  For $m,n\ge0$, we have
  \[
    \LL(P_m(x) P_n(x)) = \sum_{\pi} \wt(\pi),
  \]
  where the sum is over all Motzkin-Schr\"oder paths from $(0,m)$ to $(n,0)$ 
such that the last $n$ steps consist only of vertical down steps
  and diagonal down steps.
\end{cor}

\begin{exam}
If $(n,m,\ell)=(0,1,1)$, we have
  \[
    \LL(P_1(x)P_1(x)) = a_1a_2 + a_1b_1+\lambda_1.
  \]
  One path is eliminated: the path $VH$, because the last step is a horizontal
  step, which is neither a vertical down step nor a diagonal down step.

If $(n,m,\ell)=(0,2,1)$, we have
  \[
    \LL(P_2(x)P_1(x)) = a_1^2a_2 + a_1a_2^2 + a_1a_2a_3 + a_1a_2b_1
    + a_1a_2b_2+ a_2\lambda_1 + a_1\lambda_2.
  \]
One path is eliminated: the path $VVH$, because the last step is a horizontal 
step, which is neither a vertical down step nor a diagonal down step.
\end{exam}

\subsection{The moments $\nu_{n,m}$}
\aftersubsection

We will find a recurrence for $\nu_{n,m}= \LL\left(x^n/d_m(x)\right)$ and a generating function for them.
We do not have a combinatorial interpretation for them.

For $m\ge1$, define $U_m(x)$ to be the quotient of $P_m(x)$ when divided by $a_mx+\lambda_m$:
  \[
    P_m(x) = (a_m x+\lambda_m) U_m(x) +P_m(-\lambda_m/a_m).
  \]
Let $f_{m,i}$ be the coefficients of $U_m(x)$:
  \[
    U_m(x) = \sum_{i=0}^{m-1} f_{m,i} x^i.
  \]

  The following Lemma~\ref{lem:nu} with $\nu_{0,0}=1$ allows us to compute
  $\nu_{n,m}$ for all $n,m\ge0$.

\begin{lem}\label{lem:nu}
For $n,m\ge1$ we have
  \begin{align}
    \label{eq:nu_n0}
    \nu_{n,0} & =\mu_{n,0} = \mu_n,  \\
    \label{eq:nu_0m}
    \nu_{0,m} & = -\frac{1}{P_m(-\lambda_m/a_m)} \sum_{i=0}^{m-1} f_{m,i}\nu_{i,m-1},\\
    \label{eq:nu_nm}
    \nu_{n,m} &= \frac{1}{a_m} \nu_{n-1,m-1} - \frac{\lambda_{m}}{a_{m}}\nu_{n-1,m}.
  \end{align}
\end{lem}
\begin{proof}
 The first identity is immediate from the definitions of $\nu_{n,m}$ and $\mu_{n,m}$.
 The second identity follows from
 \[
   0 = \LL\left( \frac{P_m(x)}{d_m(x)}\right)
   =\LL\left( \frac{U_m(x)}{d_{m-1}(x)} + \frac{P_m(-\lambda_m/a_m)}{d_{m}(x)}\right).
 \]
 The third identity follows from
\[
    a_m \nu_{n,m}  
= \LL\left( \frac{x^{n-1}(a_mx+\lambda_m)}{d_m}-\frac{\lambda_m x^{n-1}}{d_m}\right)
=\nu_{n-1,m-1} -\lambda_m\nu_{n-1,m}.
\]
\end{proof}

For $m\ge0$, let
  \[
    V_m(x) = \sum_{n\ge0} \nu_{n,m} x^n. 
  \]

\begin{prop}\label{prop:Vm}
For an integer $m\ge1$, we have
  \[
    V_m(x) = \frac{a_m\nu_{0,m}}{a_m+\lambda_m x}+\frac{x V_{m-1}(x)}{a_m+\lambda_m x}.
  \]
\end{prop}
\begin{proof}
  By  \eqref{eq:nu_nm}, for $n\ge1$, 
  \[
    a_m \nu_{n,m} +\lambda_m \nu_{n-1,m} = \nu_{n-1,m-1}.
  \]
  By multiplying both sides by $x^n$ and summing over $n\ge1$, we obtain
  \[
    a_m(V_m(x) -\nu_{0,m}) + \lambda_m x V_{m}(x) = x V_{m-1}(x),
  \]
which is equivalent to the desired equation.
\end{proof}

By iterating the equation in Proposition~\ref{prop:Vm} 
and observing the fact $V_0(x)=\sum_{n\ge0}\mu_n x^n$
we obtain the following corollary.
\begin{cor}
For $m\ge1$, we have
  \[
    V_m(x) = \frac{x^m\sum_{n\ge0}\mu_n x^n}{\prod_{j=1}^m (a_j+\lambda_jx)}
    + \sum_{i=1}^m \frac{\nu_{0,i}x^{m-i}}{\prod_{j=i}^m (1+\lambda_jx/a_j)}.
  \]
\end{cor}

\section{Moments of Laurent biorthogonal polynomials}
\label{sec:inverted-polynomials}

In this section we study Laurent biorthogonal polynomials $P_n(x)$, which are
type $R_I$ orthogonal polynomials with $\lambda_n=0$. Kamioka \cite{Kamioka2007,
  Kamioka2008, Kamioka2014} combinatorially studied this case. There is another
linear functional $\FF$ that gives a different type of orthogonality for
$P_n(x)$. We will first study the connection between $\FF$ and our linear
functional $\LL$. We then show that Kamioka's results can be derived as
special cases of Theorem~\ref{thm:mu_nml}.

In this section we consider the case $\lambda_n=0$ for all $n\ge0$, so that the
polynomials $P_n(x)$ are defined by $P_{-1}(x)=0$, $P_0(x)=1$, and for $n\ge0$,
\begin{equation}
  \label{eq:12}
  P_{n+1}(x) = (x-b_n)P_n(x) -a_n x P_{n-1}(x).  
\end{equation}

\textbf{Throughout this section we assume that $P_n(0)\ne0$ and $a_n\ne0$ for
  all $n\ge0$.} Since $P_n(0)=(-1)^nb_0b_1\dots b_{n-1}$, we must have $b_n\ne0$
for all $n\ge0$.

For $n\ge0$, let 
\[
  Q_n(x) = \frac{P_n(x)}{a_1\dots a_n x^n},
\]
where $Q_0(x)=1$. Then $V=\mathrm{span}\{x^n Q_m(x): n,m\ge0\}$ is the vector
space of Laurent polynomials.

Zhedanov \cite[Proposition~1.2]{Zhedanov_1998} showed the following Favard-type
theorem, see also Kamioka \cite[Theorem~2.1]{Kamioka2014}.

\begin{thm}\label{thm:kamioka1}
  There is a unique linear functional $\FF$
  on $V$ such that $\FF(1)=1$ and 
\[
  \FF(x^{-n}P_m(x)) = 0, \qquad 0\le n<m.
\]
\end{thm}

\begin{remark}
  We note that the original statement in \cite[Proposition~1.2]{Zhedanov_1998}
  and \cite[Theorem~2.1]{Kamioka2014} is that
there is a unique linear functional $\FF$ on $V$ such that $\FF(1)=1$ and
\[
  \FF(x^{-n}P_m(x)) = h_n\delta_{n,m}, \qquad 0\le n\le m,
\]
for some constants $h_n\ne 0$. Since $\Span(\{1\}\cup\{x^{-n}P_m(x):0\le
n<m\})=V$ and $x^{-k}P_k(x)$, for $k\ge0$, is a linear combination of the
elements in the spanning set $\{1\}\cup\{x^{-n}P_m(x):0\le n<m\}$, where the
coefficient of $1$ is nonzero, the two statements are equivalent.
\end{remark}

Using Theorem~\ref{thm:unique L} we obtain a slightly different Favard-type
theorem.

\begin{thm}\label{thm:kamioka2}
  There is a unique linear functional $\LL$ on $V$ such that
  $\LL(1)=1$ and 
\[
  \LL(x^{-n}P_m(x)) = 0, \qquad 0< n\le m.
\]
\end{thm}
\begin{proof}
 By Theorem~\ref{thm:unique L}, there is a unique linear functional
$\LL$ on $V$ satisfying the orthogonality
\[
\LL\left(x^{n-m} P_m(x)/a_1\dots a_m\right) = 0, \qquad 0\le n< m.
\]
Replacing $n$ by $m-n$ in the above equation gives the theorem.
\end{proof}

Note that, since $P_1(x)=x-b_0$, if $n=0$ and $m=1$ in
Theorem~\ref{thm:kamioka1} we obtain
\begin{equation}
  \label{eq:F(x)}
 \FF(x)  = b_0.
\end{equation}
Similarly, if $n=m=1$ in Theorem~\ref{thm:kamioka2}, we obtain
\begin{equation}
  \label{eq:L(x)}
 \LL(x^{-1})  = b_0^{-1}.
\end{equation}

We now show that the linear functionals $\FF$ and $\LL$ in the above two
Favard-type theorems have a simple connection.

\begin{prop}\label{prop:F=L}
  For all $f(x)\in V$, we have
\[
\FF(f(x)) = b_0 \cdot \LL(x^{-1}f(x)).
\]
\end{prop}
\begin{proof}
  We will show the equivalent statement 
\begin{equation}\label{eq:LL=FF}
\LL(f(x)) = b_0^{-1} \cdot \FF(xf(x)).
\end{equation}
Let $\LL'(f(x))$ be the right hand side of \eqref{eq:LL=FF}.
Then by Theorem~\ref{thm:kamioka2} it suffices to show that
$\LL'(1)=1$ and
\[
  \LL'(x^{-n}P_m(x)) = 0, \qquad 0< n\le m.
\]
By definition of $\LL'$ and \eqref{eq:F(x)}, we have $\LL'(1)=b_0^{-1}\FF(x)=1$. For
$0<n\le m$, we have
\[
\LL'(x^{-n}P_m(x))=b_0^{-1} \cdot \FF(x^{-(n-1)}P_m(x)) =0
\]
because $0\le n-1<m$. This completes the proof.
\end{proof}

We will find another connection between the linear functionals $\FF$ and $\LL$
using inverted polynomials.

\begin{defn}
  The \emph{inverted polynomial} of $P_n(x)$ is defined by $\vP_n(x) := x^n
  P_n(x^{-1})/P_n(0)$. For any linear functional $\MM$ on $V$ define $\MM^\vee$ by
\[
   \MM^\vee(f(x)) := \MM(f(x^{-1})).
\]
\end{defn}

Using \eqref{eq:12} and $P_n(0)=(-1)^n
b_0b_1\dots b_{n-1}$, we have
\[
\vP_{n+1}(x)=(x-\vb_{n})\vP_{n}(x)-\va_{n}x \vP_{n-1}(x),
\]
where
\[
  \vb_n := \frac{1}{b_n}, \qquad
  \va_n := \frac{a_n}{b_{n-1} b_n}.
\]

It is easy to check that the map $X\mapsto X^\vee$ is an involution, i.e.,
$(X^{\vee})^{\vee}=X$, for each $X\in\{P_n,a_n,b_n,\MM\}$.

Let $P=\{P_n(x)\}_{n\ge0}$ be the sequence of polynomials given by \eqref{eq:12}. By
Theorems~\ref{thm:kamioka1} and \ref{thm:kamioka2} there are unique linear
functionals, denoted by $\LL_P$ and $\FF_P$, satisfying $\LL_P(1)=\FF_P(1)=1$ and 
  \begin{align*}
  \LL_P(x^{-n}P_m(x)) &= 0, \qquad 0< n\le m,\\
  \FF_P(x^{-n}P_m(x)) &= 0, \qquad 0\le n< m.
  \end{align*}
  We will sometimes write $\LL$ in place of $\LL_P$.

\begin{prop}\label{prop:L,F}
  Let $P=\{P_n(x)\}_{n\ge0}$ be the sequence of polynomials given by \eqref{eq:12}
  and let $P^\vee=\{P^\vee_n(x)\}_{n\ge0}$. Then we have
\[
\FF_{P^\vee} = \LL_P^\vee, \qquad \LL_{P^\vee} = \FF_P^\vee.
\]
\end{prop}

\begin{proof}
For $0\le n<m$, by the definition of $\vP_m(x)$, we have
\[
\wLL_P(x^{-n}\vP_m(x)) =\LL_P(x^{n}\vP_m(x^{-1})) =\LL_P(x^{-(m-n)} P_m(x)/P_m(0)) 
= 0,
\]
where the last equality follows from Theorem~\ref{thm:kamioka2} since $0<m-n\le
m$. By Theorem~\ref{thm:kamioka1}, this shows the first identity $\FF_{P^\vee} =
\LL_P^\vee$. 

Applying the first identity to $P^\vee$, we have
$\LL_{P^\vee}^\vee = \FF_{(P^\vee)^\vee}$.
Then
\[
\LL_{P^\vee} = (\LL_{P^\vee}^\vee)^\vee = (\FF_{(P^\vee)^\vee})^\vee = \FF_P^\vee,
\]
which gives the second identity.
\end{proof}

Now we show that the special case \( \lambda_n=0 \) of Theorem~\ref{thm:mu_nml}
implies the following theorem, which is equivalent to Kamioka's result \cite[Theorem~17]{Kamioka2007}.

\begin{thm}\label{thm:kamioka}
For $n,m,\ell\ge0$, we have
\begin{align}
  \label{eq:k1}
  \LL(x^nP_m(x)Q_\ell(x)) &=  \sum_{\pi\in\Sch((0,m)\to(n,\ell))} \wt(\pi),\\
  \label{eq:k2}
  \LL(x^{-n-1}Q_m(x)P_\ell(x)) &=
  \frac{a_1^\vee\dots a_\ell^\vee P_m(0) P_\ell(0)}{a_1\dots a_m b_0}
  \sum_{\pi\in\Sch((0,m)\to(n,\ell))} \wt^\vee(\pi),
\end{align}
where $\vwt(\pi)$ is the same weight $\wt(\pi)$ with $b_n$ and $a_n$ replaced by
$\vb_n$ and $\va_n$, respectively.
\end{thm}
\begin{proof}
  The first identity \eqref{eq:k1} is the special case $\lambda_n=0$ of
  Theorem~\ref{thm:mu_nml}.

  For the second identity, since $Q_\ell(x)= x^{-\ell}P_\ell(x)/a_1\dots
  a_\ell$, applying \eqref{eq:k1} to the inverted polynomials $P^\vee_k(x)$, we
  obtain
\begin{equation}
  \label{eq:15}
  \LL_{P^\vee}\left(x^nP^\vee_m(x) \frac{x^{-\ell}P^\vee_\ell(x)}{a^\vee_1\dots a^\vee_\ell} \right)
= \sum_{\pi\in\Sch((0,m)\to(n,\ell))} \vwt(\pi).
\end{equation}
Observe that
\[
\LL_{P^\vee}(f(x)) = \FF_P^\vee(f(x)) = \FF_P(f(x^{-1})) = b_0 \LL_P(x^{-1}f(x^{-1})),
\]
where the first, second, and third equalities follow from
Proposition~\ref{prop:L,F}, the definition of $\FF_P^\vee$, and
Proposition~\ref{prop:F=L}, respectively.
Therefore the left hand side of \eqref{eq:15} can be rewritten as
\begin{equation}\label{eq:15'}
    b_0\LL_{P}\left(x^{-1-n}P^\vee_m(x^{-1}) \frac{x^{\ell}P^\vee_\ell(x^{-1})}{a^\vee_1\dots a^\vee_\ell} \right)
  = \frac{a_1\dots a_m b_0}{a_1^\vee\dots a_\ell^\vee P_m(0) P_\ell(0)}
    \LL\left(x^{-1-n}Q_m(x) P_\ell(x) \right),
\end{equation}
where the following identities are used:
\[
  P^\vee_m(x^{-1}) = \frac{x^{-m} P_m(x)}{P_m(0)} = \frac{a_1\dots a_m Q_m(x)}{ P_m(0)} ,
  \qquad P^\vee_\ell(x^{-1}) = \frac{x^{-\ell} P_\ell(x)}{P_\ell(0)}.
\]
Using \eqref{eq:15} and \eqref{eq:15'} we obtain the second identity.
\end{proof}

If $m=\ell=0$ in Theorem~\ref{thm:kamioka} we obtain the following corollary.

\begin{cor}\cite[Theorem~3.1]{Kamioka2014}
  For $n\ge0$, we have
  \begin{align*}
\LL(x^{n}) &=  \sum_{\pi\in\Sch_n} \wt(\pi),\\
\LL(x^{-n-1}) &= b_0^{-1} \sum_{\pi\in\Sch_n} \vwt(\pi).
  \end{align*}
\end{cor}

\section{Lattice paths with bounded height and continued fractions}
\label{sec:lattice-paths-with}

In this section we express the generating function for Motzkin-Schr\"oder paths
with bounded height as quotients of inverted polynomials where the indices of
the sequences $b=\{b_n\}_{n\ge0}, a=\{a_n\}_{n\ge0}$, and
$\lambda=\{\lambda_n\}_{n\ge0}$ are shifted.

Recall that $P_n(x)$ are defined in \eqref{eq:favard1}. We will denote this
polynomial by $P_n(x;b,a,\lambda)$ to indicate that the three-term recurrence
coefficients are taken from the sequences $b=\{b_n\}_{n\ge0},
a=\{a_n\}_{n\ge0}$, and $\lambda=\{\lambda_n\}_{n\ge0}$. The inverted polynomial
$P_n^*(x) = x^n P_n(1/x)$ will also be written as $P_n^*(x;b,a,\lambda) = x^n
P_n(1/x;b,a,\lambda)$. Note that $P_n^*(x) = x^n P_n(1/x)$ satisfy
$P^*_{-1}(x)=0, P^*_1(x)=1$ and
\begin{equation}
  \label{eq:rec*}
P^*_{n+1}(x)=(1-b_{n}x)P^*_{n}(x)-(a_{n}x+\lambda_nx^2) P^*_{n-1}(x).  
\end{equation}

\begin{defn}
For a sequence $s=\{s_n\}_{n\ge 0}$ define $\delta s = \{s_{n+1}\}_{n\ge 0}$. 
For $P_n(x)=P_n(x;b,a,\lambda)$, we also define
\[
  \delta P_n(x) = \delta P_n(x;b,a,\lambda) = P_n(x;\delta b,\delta a,\delta \lambda),
\]
\[
  \delta P^*_n(x) = \delta P^*_n(x;b,a,\lambda) = P^*_n(x;\delta b,\delta a,\delta \lambda).
\]
\end{defn}

\begin{defn}
  We denote by $\MS_{n,r,s}^{\le k}$ the set of Motzkin-Schr\"oder paths from
  $(0,r)$ to $(n,s)$ such that the $y$-coordinate of every point is at most $k$.
  Define $\MS_{n}^{\le k}:=\MS_{n,0,0}^{\le k}$ and 
\[
  \mu^{\le k}_{n,r,s} := \sum_{\pi\in \MS_{n,r,s}^{\le k}} \wt(\pi),
\]
\[
  \mu^{\le k}_{n} :=  \sum_{\pi\in \MS_n^{\le k}} \wt(\pi).
\]
\end{defn}

The goal of this section is to prove the following theorem, which is a
generalization of Viennot's results \cite[(27) on page V-19]{ViennotLN} on
orthogonal polynomials (the case $a_n=0$).

\begin{thm}\label{thm:trunc}
Let \( r,s, \) and \( k \) be integers with \( 0\le r,s\le k \).
If $r\le s$, then 
\begin{equation}
  \label{eq:trunc2}
    \sum_{n\ge0} \mu_{n,r,s}^{\le k} x^n = \frac{P^*_r(x) \delta^{s+1} P^*_{k-s}(x)}{P^*_{k+1}(x)} \cdot x^{s-r}.  
\end{equation}
  If $r> s$, then
  \begin{equation}
    \label{eq:trunc3}
    \sum_{n\ge0} \mu_{n,r,s}^{\le k} x^n = \frac{ P^*_s(x) \delta^{r+1} P^*_{k-r}(x)}{P^*_{k+1}(x)}
    \cdot \prod_{i=s+1}^{r} (a_i+\lambda_i x).
  \end{equation}
\end{thm}

If $r=s=0$ in Theorem~\ref{thm:trunc} we obtain the following corollary. 

\begin{cor}\label{cor:mu<k}
  For a nonnegative integer \( k \), we have
  \[
    \sum_{n\ge0} \mu_n^{\le k} x^n = \frac{\delta P^*_k(x)}{P^*_{k+1}(x)}.
  \]
\end{cor}

On the other hand, using Flajolet's argument \cite{Flajolet1980}, we obtain a
continued fraction expression for the generating function for $\mu_n^{\le k}$.

\begin{prop}\label{prop:mu=cont}
  For a nonnegative integer \( k \), we have
  \[
 \sum_{n\ge0} \mu_n^{\le k}x^n = 
    \cfrac{1}{
      1-b_0x- \cfrac{a_1x+\lambda_1x^2}{
        1-b_1x- \cfrac{a_2x+\lambda_2x^2}{
          1-b_2x- \genfrac{}{}{0pt}{1}{}{\displaystyle\ddots -
            \cfrac{a_kx+\lambda_kx^2}{1-b_kx}}} }}.
  \]
\end{prop}

\begin{remark}
Combining Corollary~\ref{cor:mu<k} and Proposition~\ref{prop:mu=cont} gives
\begin{equation}\label{eq:P/P=cont}
    \frac{\delta P^*_k(x)}{P^*_{k+1}(x)} = \cfrac{1}{
      1-b_0x- \cfrac{a_1x+\lambda_1x^2}{
        1-b_1x- \cfrac{a_2x+\lambda_2x^2}{
          1-b_2x- \genfrac{}{}{0pt}{1}{}{\displaystyle\ddots -
            \cfrac{a_kx+\lambda_kx^2}{1-b_kx}}} }},
\end{equation}
which can also be shown using the fundamental recurrence relations for
continued fractions, see \cite[Chapter III, \S2]{Chihara}.
\end{remark}

For the remainder of this section we give a proof of Theorem~\ref{thm:trunc}. To
do this we give a combinatorial meaning to
\[
P^*_{k+1}(x) \sum_{n\ge0} \mu_{n,r,s}^{\le k} x^n .
\]

First we need a combinatorial interpretation for $P^*_n(x)$. Similarly to
Theorem~\ref{thm:favard}, the recurrence \eqref{eq:rec*} gives the following
proposition.

\begin{prop}\label{prop:favard*}
For $n\ge0$, we have
  \[
    P^*_n(x) = \sum_{T\in\FT_n} \wt(T)x^{\rm(T)+\bd(T)+2\rd(T)}.
  \]
\end{prop}

Let
  \[
    \MS_{*,r,s}^{\le k} := \bigcup_{n\ge0} \MS_{n,r,s}^{\le k}.
  \]
  For $\pi\in \MS_{*,r,s}^{\le k}$, we define $|\pi|$ to be $n$ if $\pi\in
  \MS_{n,r,s}^{\le k}$.
  Then by Proposition~\ref{prop:favard*} we have
\[
   P^*_{k+1}(x) \sum_{n\ge0} \mu_{n,r,s}^{\le k} x^n  =
 \sum_{T\in\FT_{k+1}} \sum_{\pi\in \MS_{*,r,s}^{\le k}} \wt(T)\wt(\pi) x^{|\pi|+\rm(T)+\bd(T)+2\rd(T)}.
\]

In what follows we construct a sign-reversing involution on $\MS_{*,r,s}^{\le
  k}\times \FT_{k+1}$, which cancels many terms in the above equation. The basic
idea is similar to that in Section~\ref{sec:moments-ops-type}: for $(\pi,T)\in
\MS_{*,r,s}^{\le k}\times \FT_{k+1}$ we add or remove a horizontal step, a peak
($(U,V)$ or $(U,D)$), or a valley ($(V,U)$ or $(D,U)$) in $\pi$ and modify the
corresponding tile(s) in $T$. 

\textbf{From now on we assume \( r\le s \).} The case \( r>s \) can be done similarly. We
first need several terminologies.

Let $\pi\in \MS_{*,r,s}^{\le k}$. A \emph{valley} of $\pi$ is a pair $(D,U)$ or
$(V,U)$ of consecutive steps in $\pi$ and a \emph{peak} of $\pi$ is a pair
$(U,D)$ or $(U,V)$ of consecutive steps in $\pi$.

Let $(\pi,T)\in \MS_{*,r,s}^{\le k}\times \FT_{k+1}$ and write $\pi=S_1\dots S_m$ as a sequence of steps. A
\emph{removable point} of $\pi$ is a point $(j,h)$ on $\pi$ satisfying one of the following conditions:
\begin{itemize}
\item $\pi$ has a horizontal step ending at $(j,h)$,
\item $h\ge r$ and $\pi$ has a peak ending at $(j,h)$, or
\item $h\le r$ and $\pi$ has a valley ending at $(j,h)$.
\end{itemize}
In other words, a removable point of $\pi$ is the ending point of a horizontal
step, a peak above the line $y= r$, or a valley below the line
$y= r$. Let $\rem(\pi)$ denote the smallest integer $i\ge1$ such that the
ending point of $S_1\dots S_i$ is a removable point. If there is no such integer
$i$, we define $\rem(\pi)=\infty$. See Figure~\ref{fig:removable1}.

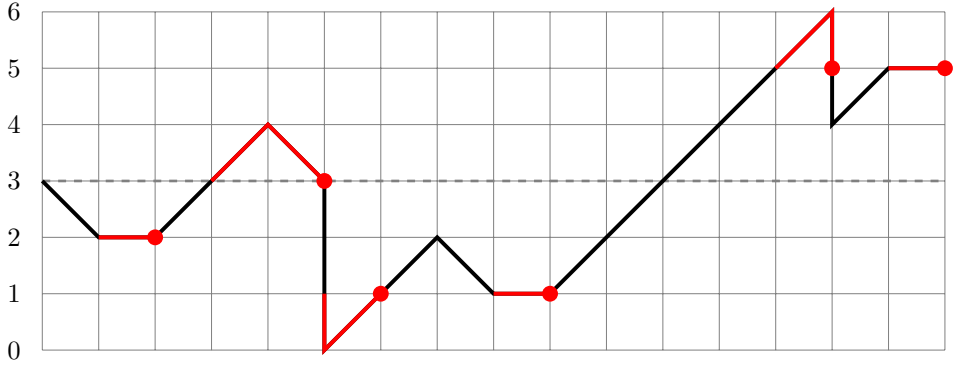
\begin{figure}
  \centering
\begin{tikzpicture}[scale=0.75]
    \draw[help lines] (0,0) grid (16,6);
    \foreach \y in {0,...,6} \draw node at (-.5,\y) {$\y$};
\draw[line width = 1pt, dashed, gray] (0,3) -- (16,3);
    \draw[line width = 1.5pt] (0,3)-- ++(1,-1)-- ++(1,0)-- ++(1,1)-- ++(1,1)-- ++(1,-1)-- ++(0,-1)-- ++(0,-1)-- ++(0,-1)-- ++(1,1)-- ++(1,1)-- ++(1,-1)-- ++(1,0)-- ++(1,1)-- ++(1,1)-- ++(1,1)-- ++(1,1)-- ++(1,1)-- ++(0,-1)-- ++(0,-1)-- ++(1,1)-- ++(1,0);
\remove(2,2) \remove(5,3)   \remove(6,1) \remove(9,1) \remove(14,5) \remove(16,5) 
\draw[line width = 1.5pt, red] (1,2) -- (2,2);
\draw[line width = 1.5pt, red] (3,3)--(4,4) -- (5,3);
\draw[line width = 1.5pt, red] (5,1) -- (5,0) -- (6,1);
\draw[line width = 1.5pt, red] (8,1) -- (9,1);
\draw[line width = 1.5pt, red] (13,5) -- (14,6) -- (14,5);
\draw[line width = 1.5pt, red] (15,5) -- (16,5);
\end{tikzpicture}  
\caption{A Motzkin-Schr\"oder path $\pi$ in $\MS_{*,r,s}^{\le k}$ for $r=3,s=5$ and $k=6$.  The dashed line is the line $y= r$.  The red dots are the removable points of $\pi$.  The horizontal step, a peak, or a valley
  ending at each removable point is colored red. In this case $\rem(\pi)=2$
  because the first removable point is the ending point of \( S_1S_2 \).}
  \label{fig:removable1}
\end{figure}

An \emph{addable point} of $(\pi,T)$ is a point $(j,h)$ on $\pi$ satisfying one of the
following conditions:
\begin{itemize}
\item $T$ has a red monomino containing $h+1$,
\item $h\ge r$ and $T$ has a (red or black) domino containing $h+1,h+2$, or
\item $h\le  r$ and $T$ has a (red or black) domino containing $h,h+1$.
\end{itemize}
In other words, an addable point of $(\pi,T)$ is an intersection of $\pi$ with
the line $y=h$ for some $h$ such that $T$ has a red monomino with $h+1$, a
domino with $h+1,h+2$ and $h\ge r$, or a domino with $h,h+1$ and $h\le
 r$. Let $\add(\pi,T)$ denote the smallest integer $i\ge0$ such that the
ending point of $S_1\dots S_i$ is an addable point. (If $i=0$, the ending point
of $S_1\dots S_i$ means the starting point of $\pi$, which is $(0,r)$.) If there
is no such integer $i$, we define $\add(\pi,T)=\infty$. See
Figure~\ref{fig:addable1}.

\begin{figure}
  \centering
\begin{tikzpicture}[scale=0.75]
    \draw[help lines] (0,0) grid (8,4);
    \foreach \y in {0,...,4} \draw node at (-.5,\y) {$\y$};
    \draw[line width = 1pt, dashed, gray] (0,2) -- (8,2);
    \draw[line width = 1.5pt] (0,2)--(0,1)-- ++(1,-1)-- ++(1,1)-- ++(1,1)-- ++(1,1)-- ++(1,1)-- ++(1,0)-- ++(0,-1)--
    ++(0,-1)-- ++(1,-1)-- ++(1,1) --++(0,1);
    \addable(0,2) \addable(3,2) \addable(4,3) \addable(6,3) \addable(6,2) \addable(1,0) \addable(8,2) \addable(8,3)
  \remove(2,1) \remove(6,4) \remove(8,2)
\end{tikzpicture}   \qquad
\begin{tikzpicture}[scale=0.75]
  \draw [help lines] (0,0) grid (6,1);
  \foreach \x in {1,...,6} \draw node at (\x-0.5,0.5) {\x};
  \LRM0 \LBD2 \LRD4 \LBM5
\end{tikzpicture}
\caption{An element $(\pi,T)\in \MS_{*,r,s}^{\le k}\times \FT_{k+1}$ for $r=2,s=3,k=5$.
  The dashed line is the line $y= r$. The removable points of $\pi$ are the red dots and $\rem(\pi)=3$.
The addable points of $(\pi,T)$ are circled.
Since the first addable point occurs at the beginning of $\pi$, we have $\add(\pi,T)=0$.}
  \label{fig:addable1}
\end{figure}
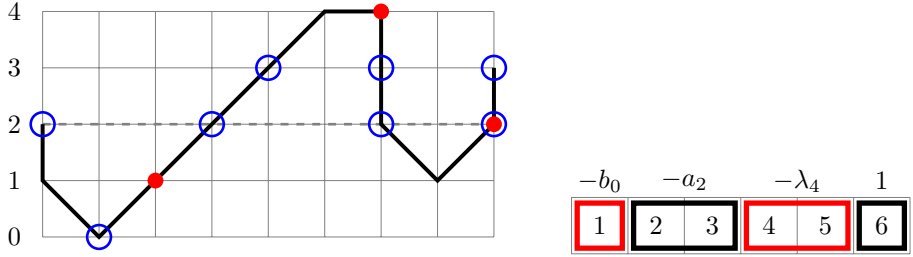

It is easy to see that if \( \rem(\pi), \add(\pi,T)<\infty \), then \(
\rem(\pi)\ne\add(\pi,T) \). For if \( \rem(\pi)=\add(\pi,T)=i \), then \( S_i \)
is a horizontal step or \( (S_{i-1},S_i) \) is a peak or a valley. But this
would imply \( \add(\pi,T)\le i-1 \) or \( \add(\pi,T)\le i-2 \)
since the ending point of \( S_1\dots S_{i-1} \) or \( S_1\dots S_{i-2} \) has
the same height as that of \( S_1\dots S_i \).

We are now ready to define a map $\phi:\MS_{*,r,s}^{\le k}\times \FT_{k+1} \to \MS_{*,r,s}^{\le k}\times \FT_{k+1}$.
Let $(\pi,T)\in \MS_{*,r,s}^{\le k}\times \FT_{k+1}$ and write $\pi=S_1\dots S_m$ as a sequence of steps.
Then $\phi(\pi,T)=(\pi',T')$ is defined as follows. 

\begin{description}
\item[Case 1] $\rem(\pi) =\add(\pi,T) =\infty$. In this case, $(\pi',T')=(\pi,T)$. 
\item[Case 2] $\rem(\pi) > \add(\pi,T)$.  Suppose that $i=\add(\pi,T)$ and $S_1\dots S_i$ ends at $(j,h)$, which is an
  addable point. Let $\tau$ be the tile of $T$ containing $h+1$.  Then 
    \[
      \pi' =
      \begin{cases}
        S_1\dots S_i H S_{i+1} \dots S_m, & \mbox{if $\tau$ is a red monomino,}\\
        S_1\dots S_i UD S_{i+1} \dots S_m, & \mbox{if $\tau$ is a red domino and $h\ge r$,}\\
        S_1\dots S_i UV S_{i+1} \dots S_m, & \mbox{if $\tau$ is a black domino and $h\ge r$,}\\
        S_1\dots S_i DU S_{i+1} \dots S_m, & \mbox{if $\tau$ is a red domino and $h\le r$,}\\
        S_1\dots S_i VU S_{i+1} \dots S_m, & \mbox{if $\tau$ is a black domino and $h\le r$,}
      \end{cases}
    \]
    and $T'$ is the tiling obtained from $T$ by replacing $\tau$ by one or two black monomino(s) according to the size
    of $\tau$.
  \item[Case 3] $\rem(\pi) < \add(\pi,T)$.  Suppose that $i=\rem(\pi,T)$ and $S_1\dots S_i$ ends at $(j,h)$, which is
    a removable point.  Then 
    \begin{align*}
      \pi' &=
      \begin{cases}
       S_1\dots \widehat{S}_{i}\dots S_m, & \mbox{if $S_{i}=H$,}\\
       S_1\dots \widehat{S}_{i-1}\widehat{S}_{i}\dots S_m, & \mbox{otherwise,}\\
      \end{cases}\\
      T' &=
      \begin{cases}
       T-B_{h+1}+R_{h+1}, & \mbox{if $S_{i}=H$,}\\
       T-B_{h+1}-B_{h+2}+R_{h+1,h+2}, & \mbox{if $(S_{i-1},S_{i})=(U,D)$ and $h\ge r$,}\\
       T-B_{h+1}-B_{h+2}+B_{h+1,h+2}, & \mbox{if $(S_{i-1},S_{i})=(U,V)$ and $h\ge r$,}\\
       T-B_{h}-B_{h+1}+R_{h,h+1}, & \mbox{if $(S_{i-1},S_{i})=(D,U)$ and $h\le r$,}\\
       T-B_{h}-B_{h+1}+B_{h,h+1}, & \mbox{if $(S_{i-1},S_{i})=(D,V)$ and $h\le r$.}\\
      \end{cases}
    \end{align*}
    Here, for example, $T-B_h-B_{h+1}+R_{h,h+1}$ means the tiling obtained from
    $T$ by removing a black monomino with $h$ and a black monomino with $h+1$
    and adding a red domino with $h,h+1$.
\end{description}

\begin{figure}
  \centering
\begin{tikzpicture}[scale=0.75]
    \draw[help lines] (0,0) grid (9,4);
    \foreach \y in {0,...,4} \draw node at (-.5,\y) {$\y$};
    \draw[line width = 1pt, dashed, gray] (0,2) -- (9,2);
    \draw[line width = 1.5pt] (0,2)--++(0,-1)--++(1,1)--++(0,-1)--++(1,-1)-- ++(1,1)-- ++(1,1)-- ++(1,1)-- ++(1,1)-- ++(1,0)-- ++(0,-1)--
    ++(0,-1)-- ++(1,-1)-- ++(1,1) --++(0,1);
    \addable(5,3) \addable(7,3)\addable(2,0) \addable(9,3)
  \remove(1,2) \remove(3,1) \remove(7,4) \remove(9,2)
\end{tikzpicture}   \qquad
\begin{tikzpicture}[scale=0.75]
  \draw [help lines] (0,0) grid (6,1);
  \foreach \x in {1,...,6} \draw node at (\x-0.5,0.5) {\x};
  \LRM0 \LBM1 \LBM2 \LRD4 \LBM5
\end{tikzpicture}
\caption{An element $(\pi,T)\in \MS_{*,r,s}^{\le k}\times \FT_{k+1}$ for $r=2,s=3,k=5$.
  The dashed line is the line $y= r$. The removable points of $\pi$ are the red dots and $\rem(\pi)=2$.
The addable points of $(\pi,T)$ are circled and $\add(\pi,T)=4$.}
  \label{fig:addable3}
\end{figure}
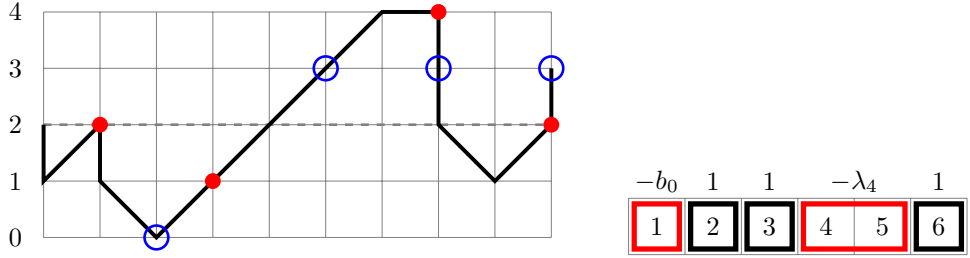

\begin{lem}\label{lem:well-defined}
  The map \( \phi \) is well defined.
\end{lem}
\begin{proof}
  The only nontrivial case is the well-definedness of \( T' \) in Case 3.

  Suppose we are in Case 3 and \( S_{i}=H \). Since \(
  i=\rem(\pi,T)<\add(\pi,T) \), the ending point \( (j,h) \) of \( S_1\dots S_i
  \) is not an addable point. In order to show that \( T' \) is well defined, we
  must show that \( B_{h+1}\in T \). By the definition of an addable point, \( T
  \) cannot have a red monomino containing \( h+1 \) because otherwise \( (j,h)
  \) would be an addable point. Thus it suffices to show that \( h+1 \) is not
  contained in any domino. We consider the three cases \( h=r \), \( h<r \), and
  \( h>r \).

  If \( h=r \), then since \( h\ge r \), \( T \) does not contain any domino
  with \( h+1,h+2 \) and since \( h\le r \), \( T \) does not contain any domino
  with \( h, h+1 \). Thus \( h+1 \) is not contained in any domino.

  If \( h>r \), then \( T \) does not contain any domino with \( h+1,h+2 \).
  Suppose that \( T \) contains a domino with \( h,h+1 \). Then since \( h-1\ge
  r \) and the line \( y=h-1 \) intersects \( \pi \), we must have \(
  \add(\pi,T)<i=\rem(\pi) \), which is a contradiction. Thus \( T \) does not
  contain a domino with \( h,h+1 \) either, and \( h+1 \) is not contained in
  any domino.

  The case \( h<r \) can be proved similarly as above, completing the proof of
  the well-definedness of \( T' \) in Case 3 when \( S_{i}=H \).

  Suppose now that $(S_{i-1},S_{i})=(U,D)$ and $h\ge r$ in Case 3. Note
  that in this case since \( i=\rem(\pi,T)<\add(\pi,T) \), the ending points \(
  (j,h) \) and \( (j-1,h+1) \) of \( S_1\dots S_i \) and \( S_1\dots S_{i-1} \)
  are not addable points. In order to show that \( T' \) is well defined, we
  must show that \( B_{h+1},B_{h+2}\in T \). By the same argument above we
  obtain that \( B_{h+1}\in T \). Moreover, since \( (j-1,h+1) \) is not an
  addable point, \( h+2 \) is not contained in a red monomino or a domino with
  \( h+2,h+3 \). Since \( h+1 \) is in a black monomino, there is no domino with
  \( h+1, h+2 \). Therefore \( h+2 \) must be in a black monomino, and we obtain
  \( B_{h+1},B_{h+2}\in T \) as desired.

  The remaining cases can be proved similarly.
\end{proof}

For example, if $(\pi,T)$ is the element in Figure~\ref{fig:addable1}, then
$\phi(\pi,T)$ is the element in Figure~\ref{fig:addable3}.

\begin{lem}\label{lem:involution}
The map $\phi:\MS_{*,r,s}^{\le k}\times \FT_{k+1} \to \MS_{*,r,s}^{\le k}\times \FT_{k+1}$ is a sign-reversing
involution, i.e., if $\phi(\pi,T)=(\pi',T')$ with $(\pi,T)\ne(\pi',T')$, then
$\wt(\pi',T')=-\wt(\pi,T)$, where
\[
\wt(\pi,T):=\wt(T)\wt(\pi) x^{|\pi|+\rm(T)+\bd(T)+2\rd(T)}.
\]
Moreover, the set of fixed points of $\phi$ is given by
\[
  \Fix(\phi) =\{(\pi,T): \rem(\pi)= \add(\pi,T)=\infty\}.
\]
\end{lem}
\begin{proof}
  This is a straightforward verification using the definition of $\phi$. We omit
  the details.
\end{proof}

Now we are ready to prove Theorem~\ref{thm:trunc}.

\begin{proof}[Proof of Theorem~\ref{thm:trunc}]
The theorem can be reformulated as follows:
\begin{equation}
  \label{eq:Pmu}
  P^*_{k+1}(x)    \sum_{n\ge0} \mu_{n,r,s}^{\le k} x^n = 
  \begin{cases}
  P^*_r(x) \delta^{s+1} P^*_{k-s}(x)  x^{s-r},& \mbox{if $r\le s$},\\
  P^*_s(x) \delta^{s+1} P^*_{k-r}(x)  \prod_{i=s+1}^{r} (a_i+\lambda_i x),& \mbox{if $r> s$}.
  \end{cases}
\end{equation}

We first consider the case $r\le s$. By Lemma~\ref{lem:involution} we have
\begin{align*}
   P^*_{k+1}(x) \sum_{n\ge0} \mu_{n,r,s}^{\le k} x^n  &=
 \sum_{T\in\FT_{k+1}} \sum_{\pi\in \MS_{*,r,s}^{\le k}} \wt(T)\wt(\pi) x^{|\pi|+\rm(T)+\bd(T)+2\rd(T)}\\
&= \sum_{(\pi,T)\in\Fix(\phi)} \wt(T)\wt(\pi) x^{|\pi|+\rm(T)+\bd(T)+2\rd(T)},
\end{align*}
where $\Fix(\phi)$ is the set of pairs $(\pi,T)\in \MS_{*,r,s}^{\le k}\times \FT_{k+1}$ such that
$\rem(\pi)= \add(\pi,T)=\infty$.
Suppose $(\pi,T)\in\Fix(\phi)$. It is easy
to see that there is a unique $\pi$ satisfying $\rem(\pi)=\infty$, namely
$\pi=UU\dots U$ consisting of $s-r$ up steps. This $\pi$ contributes the factor
$x^{s-r}$ in \eqref{eq:Pmu}. Moreover, since $\add(\pi,T)=\infty$, we must have
that the tile in $T$ containing $h$ must be a black monomino for every $r+1\le
h\le s+1$. On the other hand, there is no restriction on the tiles containing
$i\le r$ and $j\ge s+2$. If we sum over all $(\pi,T)\in \Fix(\phi)$, the part of
$T$ consisting of tiles with entries in $\{1,2,\dots,r\}$
(resp.~$\{s+2,s+3,\dots,k+1\}$) contributes the factor $P_r^*(x)$
(resp.~$\delta^{s+1} P^*_{k-s}(x)$) in \eqref{eq:Pmu}.
This shows \eqref{eq:Pmu} when $r\le s$.

Now we consider the case $r>s$. This can be shown similarly as in the previous
case by modifying the map \( \phi \) so that we read the path \( \pi \)
backwards \( \pi=S_m\dots S_1 \). Note that in this case $\pi$ is not unique in
\( \Fix(\phi) \), but $\pi$ can be any path with $r-s$ steps such that each step is
either $U$ or $V$. The contribution of such $\pi$'s is the factor
$\prod_{i=s+1}^{r} (a_i+\lambda_i x)$ in \eqref{eq:Pmu}. This completes the
proof.
\end{proof}

\section{Determinants for type $R_I$ polynomials}  
\label{sec:hank-determ-type}

The classical orthogonal polynomials $p_n(x)=\sum_{j=0}^n c_j x^j$, $c_n\neq 0,$ are 
uniquely defined if there is a unique solution to 
the linear equations, for each $n$,  
\begin{align*}
\LL(x^ip_n(x))&=0, \quad  i=0,\dots, n-1\\
 \LL(x^np_n(x))&=K_n\ne 0.
\end{align*}
These linear equations for the coefficients $c_i$ are 
$$
\sum_{j=0}^{n} \mu_{i+j}c_j =K_n \delta_{i,n}, \  i=0,\dots,n.
$$
Thus the appropriate condition for a unique solution is \( \Delta_{n}\neq 0 \), where \(
\Delta_n=\det(\mu_{i+j})_{i,j=0}^n \) is the Hankel determinant of the system.
Moreover, Cramer's rule gives the following expression \cite[p.~12]{Chihara}.
  \begin{equation}\label{eq:cramer}
    p_n(x) =\frac{K_n}{\det(\mu_{i+j})_{i,j=0}^{n}}\det
    \begin{pmatrix}
      (\mu_{i+j})_{\substack{0\le i \le n-1\\ 0\le j \le n}}\\ (x^j)_{0\le j\le n}
    \end{pmatrix}.
  \end{equation}

  If the three-term recurrence for monic \( p_n(x) \) is given by \(
  p_{n+1}(x)=(x-b_n)p_n(x)-\lambda_np_{n-1}(x) \), then 
  \( K_n=\lambda_1\cdots \lambda_n \) and 
  the Hankel determinant
  factors:
  \begin{equation}\label{eq:Delta_n}
    \Delta_n := \det(\mu_{i+j})_{i,j=0}^n  = \prod_{k=1}^{n}\lambda_k^{n+1-k}.
  \end{equation}
  Therefore we can rewrite \eqref{eq:cramer} as the following theorem, which may
  also be obtained from the uniqueness of \( p_n(x) \) \cite[p.~17, Exercise
  3.1]{Chihara}.

\begin{thm}
\label{thm:P=det/det}
  If $\lambda_n\neq 0$ for all $n$, then the monic orthogonal polynomials satisfy
  \[
    P_n(x) =\frac{1}{\det(\mu_{i+j})_{i,j=0}^{n-1}}\det
    \begin{pmatrix}
      (\mu_{i+j})_{\substack{0\le i \le n-1\\ 0\le j \le n}}\\ (x^j)_{0\le j\le n}
    \end{pmatrix}.
  \]
\end{thm}

The purpose of this section is to give analogues of Theorem~\ref{thm:P=det/det} and
\eqref{eq:Delta_n} for type $R_I$ polynomials. We also consider the special case
when the three-term recurrence coefficients are all constants.

\subsection{Determinants for general type \( R_I \) polynomials}
\aftersubsection

In this subsection we give analogues of Theorem~\ref{thm:P=det/det} and
\eqref{eq:Delta_n} for type $R_I$ polynomials. We consider three different
analogous determinants, and three factorizations of determinants of the system.
We use three sets 
\[
  \{b_j^{(t)}(x):j=0,1,\dots,n\}, \qquad t=1,2,3,
\]
 of linearly
independent elements of $V$ to expand the rational function $Q_n(x)$, where for
\( j=0,1,\dots,n \),
\begin{align*}
  b_j^{(1)}(x) &= x^j/d_n(x), \\
  b_j^{(2)}(x) &= x^j/d_j(x),\\
  b_j^{(3)}(x) &= 1/d_j(x).
\end{align*}

Let's consider first the case of $t=1$, and expand
$$
Q_n(x)=\sum_{j=0}^n c_j^{(1)} b_j^{(1)}(x).
$$
The defining orthogonality relations are

$$
\LL(x^iQ_n(x))=0, \  i=0,\dots, n-1, \quad \LL(x^nQ_n(x))=1,  {\text{ for all }}  n .
$$
These linear equations are 
\begin{equation}
\label{eq:detsys1}
\sum_{j=0}^{n} \nu_{i+j,n}c_j^{(1)} =\delta_{i,n}, \qquad i=0,\dots, n.
\end{equation}
The determinant of the system \eqref{eq:detsys1} is 
$$
\Delta_n'= \det(\nu_{i+j,n})_{i,j=0}^{n}.
$$

Assuming $\Delta_n'\neq0$ for all $n$, the solution is unique. 
In addition Corollary~\ref{cor:L(P_mQ_ell)} shows that the monic polynomial 
$P_n(x)$ does satisfy $\LL(x^nQ_n(x))=\LL(x^nP_n(x)/d_n(x))=1,$ so the 
Cramer's rule solution is monic.

\begin{prop}
\label{prop:PQunique}
  The \( P_n(x) \) and \( Q_n(x) \) are uniquely determined if and only if the
  following condition holds
  \begin{equation}\label{eq:Dne0}
 \Delta_n'= \det(\nu_{i+j,n})_{i,j=0}^{n} \neq 0  {\text{ for all  }} n.
  \end{equation}
\end{prop}

\begin{thm}
\label{thm:firstDelta}
Suppose that \eqref{eq:Dne0} holds. Then
the monic type $R_I$ polynomials are 
$$
P_n(x) =\frac{1}{\Delta_n'}\det
\begin{pmatrix}
  (\nu_{i+j,n})_{\substack{0\le i \le n-1\\ 0\le j \le n}}\\ (x^j)_{0\le j\le n}
\end{pmatrix},
$$
and 
$$
 \Delta_{n}'=\prod_{k=1}^{n}\frac{1}{ (-a_k)^k P_k(-\lambda_k/a_k)}.
$$
\end{thm}

\begin{proof} The first statement follows from Cramer's rule.  

  Let's assume that $\lambda_n\neq 0$ and choose $x=-\lambda_n/a_n$ in the first
  statement. Then successively multiply the $j$th column by $a_n/\lambda_n$ and
  add it to the $(j-1)$st column, for $j=1,2,\dots, n$. The resulting last row
  is $(0,0,\dots,0,(-\lambda_n/a_n)^n)$. The new $(i,j)$-entry, for \( 0\le
  i,j\le n-1 \), is
 $$
 \nu_{i+j,n}+\frac{a_n}{\lambda_n}\nu_{i+j+1,n}=
 \frac{1}{\lambda_n} \nu_{i+j,n-1},
 $$
 where we have used \eqref{eq:nu_nm}. This proves 
 $$
 P_n(-\lambda_n/a_n)=(-\lambda_n/a_n)^n \Delta_{n-1}'(\lambda_n)^{-n} /\Delta_n',
 $$
 which is the second statement. Because all of the quantities are rational functions of 
 $\lambda_n$, it also holds for $\lambda_n=0.$
\end{proof}

Note that the conditions for uniqueness agree with our assumptions, $a_n\ne0$ and
$P_n(-\lambda_n/a_n)\ne0$ for all $n\ge0$, throughout the paper.

\begin{cor} The \( P_n(x) \) and \( Q_n(x) \) are uniquely determined if
$$
a_k\neq 0 {\text{ and }}\  P_k(-\lambda_k/a_k)\neq 0 {\text{ for all }} k\ge 0.
$$
\end{cor}

The Cramer's rule solution being monic gives the following determinant identity. 
\begin{cor}
We have
\[
\Delta'_n =\det(\nu_{i+j,n})_{0\le i,j \le n}= \det(\nu_{i+j,n}) _{0\le i,j \le n-1}.  
\]
\end{cor}

Next we give the results for the remaining two bases, $b_j^{(2)}$ and $b_j^{(3)}.$
The proofs are similar, again the Cramer's rule solution is the monic solution. 
Define
\begin{align*}
  \Delta''_{n} &:=  \det(\nu_{i+j,j})_{0\le i,j \le n},\\
  \Delta'''_{n} &:=  \det(\nu_{i,j})_{0\le i,j \le n}.
\end{align*}

\begin{thm}
\label{thm:PQ=det}
Suppose that  \eqref{eq:Dne0} holds.
We have, assuming \( \lambda_k\ne 0 \) for all \( k \) for \eqref{eq:QQ=det},
\begin{equation}
  \label{eq:QQ=det}
Q_n(x)= \frac{1}{\Delta''_{n}}\det
\begin{pmatrix}
\nu_{0,0}  &\nu_{1,1} & \cdots& \nu_{n,n} \\
\nu_{1,0}  &\nu_{2,1} & \cdots& \nu_{n+1,n} \\
\vdots  &\vdots  & \ddots& \vdots \\
\nu_{n-1,0}  &\nu_{n,1} & \cdots& \nu_{2n-1,n} \\
1&\frac{x}{d_1(x)}&\cdots &\frac{x^n}{d_n(x)}
\end{pmatrix}
=\frac{1}{\Delta''_{n}}\det
  \begin{pmatrix}
\left(\nu_{i+j,j}\right)_{\substack{0\le i \le n-1\\ 0\le j \le n}}\\ \left(\frac{x^j}{d_j(x)}\right)_{0\le j\le n}
  \end{pmatrix},
\end{equation}
\begin{equation}
  \label{eq:Q=det2}
Q_n(x)= \frac{1}{\Delta'''_{n}}\det
\begin{pmatrix}
\nu_{0,0}  &\nu_{0,1} & \cdots& \nu_{0,n} \\
\nu_{1,0}  &\nu_{1,1} & \cdots& \nu_{1,n} \\
\vdots  &\vdots  & \ddots& \vdots \\
\nu_{n-1,0}  &\nu_{n-1,1} & \cdots& \nu_{n-1,n} \\
1&\frac{1}{d_1(x)}&\cdots &\frac{1}{d_n(x)}
\end{pmatrix}
=\frac{1}{\Delta'''_{n}}\det
\begin{pmatrix}
\left(\nu_{i,j}\right)_{\substack{0\le i \le n-1\\ 0\le j \le n}}\\ \left(\frac{1}{d_j(x)}\right)_{0\le j\le n}
  \end{pmatrix} .
\end{equation}
\end{thm}

\begin{thm}\label{thm:3Deltas}
  Suppose that  \eqref{eq:Dne0} holds. Then
    \begin{align}
      \label{eq:Delta''}
      \Delta''_{n}&=  \prod_{k=1}^{n}\frac{\lambda_k^k}{ (-a_k)^k P_k(-\lambda_k/a_k)},\\
      \label{eq:Delta'''}
      \Delta'''_n&=\prod_{k=1}^n\frac{1}{P_{k}(-\la_k/a_k)}.    
    \end{align}
\end{thm}

\begin{remark}
  A referee suggested using elementary column operations for
  the proof of Theorem~\ref{thm:firstDelta}, also a proof of 
  Theorem~\ref{thm:3Deltas}. Our original proofs were more
  complicated, and we thank this referee.
\end{remark}

\begin{remark}
  We can take special cases of Theorem~\ref{thm:PQ=det} to recover the classical
  results on orthogonal polynomials and Laurent biorthogonal polynomials.
  
If $a_k=0$ for all $k\ge0$, then $d_k(x) = \lambda_1\dots\lambda_k$, 
$\nu_{i,j}=\mu_i/\lambda_1\dots\lambda_j$, and therefore
\begin{equation}
  \label{eq:3Deltas_a=0}
  \Delta'_n=\frac{\Delta_n}{(\lambda_1\dots\lambda_n)^{n+1}},\qquad 
\Delta''_n=\frac{\Delta_n}{\lambda_1^{n}\lambda_2^{n-1}\dots\lambda_n^1}, \qquad
\Delta'''_n=0.
\end{equation}
In this case $P_n(x)$ become the usual orthogonal polynomials and
\eqref{eq:QQ=det} reduces to  Theorem~\ref{thm:P=det/det}.

If $\lambda_k=0$ for all $k\ge0$, then $d_k(x) = a_1\dots a_k x^k$, 
$\nu_{i,j}=\LL(x^{i-j})/a_1\dots a_j$, and
\begin{equation}
  \label{eq:3Deltas_la=0}
  \Delta'_n=\frac{\det(\LL(x^{i+j-n}))_{0\le i,j\le n}}{(a_1\dots a_n)^{n+1}},\qquad 
\Delta''_n=0, \qquad
\Delta'''_n=\frac{\det(\LL(x^{i-j}))_{0\le i,j\le n}}{a_1^n a_2^{n-1} \dots a_n^1}.
\end{equation}
In this case the polynomials $P_n(x)$ become Laurent biorthogonal polynomials
and \eqref{eq:QQ=det} reduces to
\begin{equation}
  \label{eq:11}
P_n(x) =\frac{a_1\dots a_n}{\det(\LL(x^{i-j}))_{0\le i,j\le n}}\det
\begin{pmatrix}
(\LL(x^{i-j}))_{\substack{0\le i \le n-1\\ 0\le j \le n}}\\ (x^{n-j})_{0\le j\le n}
  \end{pmatrix},
\end{equation}
see \cite[(1.2)]{Zhedanov_1998} and \cite[(2.8)]{Kamioka2014}.
\end{remark}

\begin{remark}
  Zhedanov \cite[Proposition~3]{Zhedanov1999} showed that type \( R_I \) is
  obtained from type \( R_{II} \) as a special case. Spiridonov and Zhedanov
  \cite[(1.3)]{Spiridonov2003} have a determinant for type \( R_{II} \)
  orthogonal polynomials. We believe that a special case of this gives our
  determinant \eqref{eq:Q=det2}, but we have not done any details.
  Maeda and Tsujimoto \cite[(22)]{Maeda2013} also found a determinant expression
  for type \( R_{II} \) polynomials.
\end{remark}

We now state a variation of Theorems~\ref{thm:firstDelta} and \ref{thm:3Deltas}
without proof. Let
\begin{align*}
  \Delta'_{n,s} &= \det\left( \nu_{s+i+j,s+n}\right)_{i,j=0}^{n},\\  
  \Delta''_{n,s} &= \det\left( \nu_{s+i+j,s+j}\right)_{i,j=0}^{n},\\
  \Delta'''_{n,s} &= \det\left( \nu_{i,s+j}\right)_{i,j=0}^{n}.
\end{align*}
Then $\Delta'_n = \Delta'_{n,0}$, $\Delta''_n = \Delta''_{n,0}$, and
$\Delta'''_n = \Delta'''_{n,0}$. 

\begin{thm}
  We have
  \begin{align}
  \label{eq:Delta'n-1,1}
\Delta'_{n-1,1}  &= \frac{(-1)^{\binom{n}2}P_n(0)}{\prod_{k=1}^{n} a_k^k P_k(-\lambda_k/a_k)},\\
  \label{eq:Delta''n-1,1}
\Delta''_{n-1,1} & = \frac{(-1)^{\binom{n}2}\lambda_1^1 \lambda_2^2\dots \lambda_n^{n}P_n(0)}
{\prod_{k=1}^{n} a_k^k P_k(-\lambda_k/a_k)},\\
  \label{eq:Delta'''n-1,1}
\Delta'''_{n-1,1}&  = \frac{(-1)^{n}}
{\prod_{k=1}^{n} a_k P_k(-\lambda_k/a_k)}.
  \end{align}
\end{thm}

\begin{remark}
If $\lambda_k=0$ for all $k\ge1$, then $P_n(0)=(-1)^nb_0b_1\dots b_{n-1}$ and $d_j(x)=a_1\dots a_j$.
In this case \eqref{eq:Delta'''n-1,1} reduce to 
\begin{align}
  \label{eq:Delta2*}
\det(\LL(x^{i-j-1}))_{0\le i,j\le n-1} &= \frac{(-1)^{\binom {n}2}}{a_1\dots a_n} \prod_{k=1}^{n} \left(\frac{a_k}{b_{k-1}}\right) ^{n+1-k},
\end{align}
which is equivalent to Kamioka's result \cite[2.14b]{Kamioka2014}.
\end{remark}
\subsection{Hankel determinants for constant values of \( b_n,\lambda_n\) and \( a_n \)}
\aftersubsection

Recall from \eqref{eq:Delta_n} that $\Delta_n$ factors in the classical case
($a_k=0$). In this subsection we will show in Theorem~\ref{thm:hankel} that
$\Delta_n$ always factors for type \( R_I \) if the three-term recurrence coefficients are all
constants. 

We need the following well-known lemma.

\begin{lem}\label{lem:hankel_change1}
  Let $p_k(x)$ and $q_k(x)$ be monic polynomials of degree $k$ for $0\le k\le n$.  Let
  $\LL$ be any linear functional defined on polynomials. Then we have
\[
    \det(\LL(x^{i+j}))_{i,j=0}^n = \det(\LL(p_i(x) q_j(x)))_{i,j=0}^n.
\]
\end{lem}
\begin{proof}
  Let $p_k(x)=\sum_{i=0}^kp_{k,i} x^i$ and $q_k(x)=\sum_{i=0}^kq_{k,i} x^i$.  Then $P=(p_{i,j})_{i,j=0}^n$
  is a lower unitriangular matrix and $Q=(q_{j,i})_{j,i=0}^n$ is an upper unitriangular matrix. Thus,
  letting $M=(\LL(x^{i+j}))_{i,j=0}^n$, we have
  \[
    \det(\LL(x^{i+j}))_{i,j=0}^n  = \det(PMQ)_{i,j=0}^n.
  \]
Since  the $(k,s)$-entry of $PMQ$ is
 \[
   \sum_{i=0}^n\sum_{j=0}^n \LL(p_{ki}x^i q_{sj}x^j)=\LL(p_k(x)q_s(x)),
 \]
we are done.
\end{proof}

We first show that $\Delta_n$ factors if $a_k$ and $b_k$ are constants and
$\lambda_k=0$.

\begin{lem}\label{lem:xin}
  If $a_{k}=1$, $b_k=t$, and $\lambda_{k}=0$ for all $k\ge0$, then
\begin{equation}
  \label{eq:hankel la_k=0'}  
  \Delta_n = (1+t)^{\binom {n+1}2}.
\end{equation}

  If $a_{k}=A$, $b_k=B$, and $\lambda_{k}=0$ for all $k\ge0$, then
\begin{equation}
  \label{eq:hankel la_k=0}  
  \Delta_n = (A^2+AB)^{\binom {n+1}2}.
\end{equation}
\end{lem}
\begin{proof}
  Using the Lindstr\"om--Gessel--Viennot lemma \cite{Lindstrom, LGV}, one
  obtains that the first identity \eqref{eq:hankel la_k=0'} is equivalent to the
  result of Sulanke and Xin \cite[Lemma~2.2]{Sulanke_2008} on non-intersecting
  Schr\"oder paths where each horizontal step has weight $t$. The second
  identity \eqref{eq:hankel la_k=0} then follows from the first using the fact
  that $\mu_{i}$ is a polynomial in $A$ and $B$ of degree $i$.
\end{proof}

Now we show that $\Delta_n$ factors if $b_k,a_k$, and $\lambda_k$ are constants.

\begin{thm}\label{thm:hankel}
If $a_{k}=A$, $b_k = B$, $\lambda_{k}=C$, for all $k\ge0$, then
  \[
    \Delta_n = (A^2+AB+C)^{\binom {n+1}2}.
  \]
\end{thm}
\begin{proof}
Let $y= x + C/A$ and $\widetilde{P}_n(y)=P_n(x)=P_n(y-C/A)$. Since
\[
P_{n+1}(x) = (x-B) P_n(x) - (Ax+C) P_{n-1}(x),
\]
we have
\[
\widetilde{P}_{n+1}(y) = (y-C/A-B) \widetilde{P}_n(y) -Ay \widetilde{P}_{n-1}(y).
\]
Moreover, since $d_n(x) = (Ax+C)^n$ and $\widetilde d_n(y) = A^ny^n$, we have $\widetilde
d_n(y)=d_n(x)$ and $\widetilde{Q}_n(y)= Q_n(x)$. Thus the orthogonality
\[
\LL(x^j Q_n(x)) = 0,  \qquad \mbox{if $0\le j<n$},
\]
implies that
\[
\LL(y^j \widetilde{Q}_n(y)) = 0,  \qquad \mbox{if $0\le j<n$}.
\]
Therefore $\widetilde{P}_n(y)$ are type $R_I$ orthogonal polynomials with the same linear functional $\LL$.

Now the original Hankel determinant is
\[
\det \left( \mu_{i+j}\right)_{i,j=0}^n= 
\det\left( \LL(x^{i+j}) \right)_{i,j=0}^n = \det\left( \LL(\left( y-C/A\right)^{i+j}) \right)_{i,j=0}^n.
\]
By Lemma~\eqref{lem:hankel_change1} with $p_k(y)=q_k(y)=(y-C/A)^k$, we have
\[
  \det\left( \LL(\left( y-C/A\right)^{i+j}) \right)_{i,j=0}^n = 
  \det\left( \LL(y^{i+j}) \right)_{i,j=0}^n.
\]
By \eqref{eq:hankel la_k=0}, we have
\[
  \det\left( \LL(y^{i+j}) \right)_{i,j=0}^n=(A^2+A(C/A+B))^{\binom{n+1}2},
\]
which finishes the proof.
\end{proof}

The following proposition provides another proof of Theorem~\ref{thm:hankel} via
\eqref{eq:Delta_n} because the $\lambda_n=A^2+AB+C$ are constant.

\begin{prop} Let $\mu_n$ be the moments for the type $R_I$ polynomials with 
$a_{k}=A$, $b_k=B$, and $\lambda_{k}=C$ for all \( k\ge0 \). Then $\mu_n$ are also the 
moments for classical orthogonal 
polynomials defined by  \( p_{-1}(x)=0 \), \( p_0(x)=1 \), and 
\[
  p_{n+1}(x) = (x-B_n)p_n(x) - \Lambda_np_{n-1}(x), \quad n\ge0,
\]
where
$$ 
B_0=A+B, \quad B_n=2A+B, \quad\Lambda_n= A^2+AB+C, \quad n\ge 1.
$$ 
\end{prop}

\begin{proof}[Sketch of Proof] Because the type $R_I$ coefficients are constant,
  the type $R_I$ continued fraction for the moment generating function in
  Corollary~\ref{cor:R1cfrac} satisfies a quadratic equation and may be solved.
  This also occurs in the classical case except for the $B_0$ term. Dealing with
  this term, and explicitly solving, shows that each moment generating function
  is
$$
\frac{1-Bx-\sqrt{1-4Ax-2Bx+B^2x^2-4Cx^2}}{2x(A+Cx)}.
$$ 
Note that Proposition~\ref{prop: CF} is the $A=B=C=1$ special case. 
\end{proof} 

\begin{remark}
  Note that in Theorem~\ref{thm:hankel}, the entries in the determinant of
  $\Delta_n = \det\left( \mu_{i+j}\right)_{i,j=0}^{n}$ are positive polynomials
  in $A,B$, and $C$. It would be interesting to prove this theorem using
  nonintersecting lattice paths such as the proof of the Aztec diamond theorem
  due to Eu and Fu \cite{Eu2005}. Kamioka \cite[Theorem~1.1]{Kamioka2014} found
  a $q$-version of the Aztec diamond theorem using moments of orthogonal Laurent
  polynomials when \( a_n=q^{2n-1} \) and \( b_n=t q^{2n} \). Brualdi and
  Kirkland \cite{Brualdi2005} evaluated the Hankel determinant of Schr\"oder
  numbers using J-fractions.
\end{remark}

\section{Explicit type $R_I$ polynomials}
\label{sec:explicit-type-r_i}

In this section we explain the methods for finding the explicit type $R_I$
polynomials later in Sections~\ref{glue} and \ref{powers}.

We need to define a linear functional $\LL$ on $V$. 
Recall that $V$ is the vector space whose basis is
$$
\{x^m: m\ge 0\} \cup \{1/d_n(x): n\ge 1\}.
$$
Note that we could replace the monomial $x^m$ by any set of polynomials, 
one for each degree.

In Section~\ref{glue} we start by taking a classical $\LL$ on the polynomial
part of $V$. We extend $\LL$ to the larger space $V$ by explicitly defining
$\LL(1/d_n(x)).$ Good choices of the denominator polynomials $d_n(x),$ found by
``gluing" onto a weight function representing $\LL$, allow the extension to be
explicitly defined by shifting parameters in the linear functional $\LL$. The
explicit type $R_I$ polynomials will be obtained by also shifting parameters.

As an example of this phenomenon, take the classical Jacobi polynomials on $[-1,1]$,
$$
\LL_{a,b}(f)=\frac{\Gamma(a+b+2)}{2^{a+b+1}\Gamma(a+1)\Gamma(b+1)}
\int_{-1}^1 w(x)f(x)dx, \quad w(x)=(1-x)^a(1+x)^b
$$ 
which includes the parameters $a$ and $b$. For the integral to converge for all
polynomials $f(x),$ one needs $a,b>-1.$ However one may extend these values to
all real numbers \( a,b \) with \( a+b \) not an integer by defining
$\LL_{a,b}$ on a basis
$$
\LL_{a,b}((1-x)^k)=2^k\frac{(a+1)_k}{(a+b+2)_k}, \quad k\ge 0.
$$
 
The choice of  $d_n(x)=(1+x)^n$ glues to $w(x)$ so that one defines
$$
\LL_{a,b}(1/d_n(x))= 2^{-n}\frac{(b-n+1)_n}{(a+b-n+2)_n}, \quad n\ge 1.
$$
This defines an extension for 
$\LL_{a,b}$ to $V$ by defining it on a basis, without referring to the integral. 
We shall see in Section~\ref{glue} that the type $R_I$ polynomials for this $\LL_{a,b}$ 
and $d_n(x)$ are the corresponding shifted Jacobi polynomials.

Since $\LL$ extended in this way always has the same moment sequence, we have 
equality of their moment generating functions, which are continued fractions. 
Each of the examples in Section~\ref{glue} satisfies this theorem.

\begin{thm} 
\label{weirdCFthm}
Suppose that $P_n(x)$ is a type $R_I$ orthogonal polynomial satisfying
\[
  P_{n+1}(x) = (x-b_n)P_n(x) -(a_n x+\lambda_n )P_{n-1}(x), \quad P_{-1}(x)=0, \quad P_0(x)=1,  
\]
whose linear functional extends that for an orthogonal polynomial
\[
  p_{n+1}(x) = (x-B_n)p_n(x) -\Lambda_n p_{n-1}(x),  \quad p_{-1}(x)=0, \quad p_0(x)=1.  
\]
Then  we have the formal power series equality between continued fractions
\begin{align*}
 \sum_{n\ge0} \mu_nx^n  &=
    \cfrac{1}{
      1-b_0x- \cfrac{a_1x+\lambda_1x^2}{
        1-b_1x- \cfrac{a_2x+\lambda_2x^2}{
          1-b_2x- \cdots}} }\\
            &=
 \cfrac{1}{
      1-B_0x- \cfrac{\Lambda_1x^2}{
        1-B_1x- \cfrac{\Lambda_2x^2}{
          1-B_2x-\cdots }}}.
 \end{align*}
\end{thm}

In Section~\ref{powers} our second method for explicit type $R_I$ polynomials chooses 
$d_n(x)=(1+ax)^n,$
an $n$th power, for some special orthogonal polynomials. 
This inserts $1+ax$ into the three-term recurrence \eqref{modevenodd}, 
which does change the linear functional $\LL$. 
Yet there is a relation between the type $R_I$ and classical 
moment generating functions, see Proposition~\ref{gencor} and Theorem~\ref{genthm}.

\section{Gluing}
\label{glue}

We shall show gluing works by considering the Jacobi polynomials
$P_n^{(a,b)}(x)$ on $[-1,1],$ whose weight function is $(1-x)^a(1+x)^b,$ and
$\LL_{a,b}$ defined as before. Then we use the same technique on classical
orthogonal polynomials in the subsections.

The linear functional on the vector space of polynomials is
\begin{equation}
\label{jacwt}
\LL_{a,b}(f(x))=\frac{\Gamma(a+b+2)}{2^{a+b+1}\Gamma(a+1)\Gamma(b+1)}
\int_{-1}^1(1-x)^a (1+x)^b f(x) dx,
\end{equation}
where
$$
\LL_{a,b}(1)=1, \quad \LL_{a,b}(P_n^{(a,b)}(x)P_m^{(a,b)}(x))=0, \quad m\neq n.
$$

As in Section~\ref{sec:explicit-type-r_i} we extend $\LL_{a,b}$ to the vector
space $V$ by defining $\LL_{a,b}$ on the basis elements $1/d_n(x)=1/(1+x)^n.$
Even though the integral may not exist, if $b<n,$ we may define the linear
functional by what the integral would give by shifting $b$ to $b-n$
\begin{equation}
\label{eq:L(1/d)}
\LL_{a,b}(1/d_n(x))=2^{-n}\frac{(b-n+1)_n}{(a+b+2-n)_n}.
\end{equation}

We assume that \( a+b+2 \) is not an integer so the denominator of
\eqref{eq:L(1/d)} is well defined. We tacitly make a similar assumption
in each of our examples.

\begin{prop}
We have for any polynomial $p(x)$
\begin{equation}
\label{masterthm}
\LL_{a,b}(p(x)/d_n(x))=2^{-n}\frac{(b-n+1)_n}{(a+b+2-n)_n}\LL_{a,b-n}(p(x)).
\end{equation}
\end{prop}

\begin{proof}
Write 
\begin{equation}
\label{hatethisproof}
\frac{p(x)}{d_n(x)}= t(x)+\sum_{i=1}^n \frac{c_i}{d_i(x)}
\end{equation}
for some polynomial $t(x)$ and some constants $c_i$. Assume that $b-n> -1$ and
$a> -1$. Then applying $\LL_{a,b}$ to \eqref{hatethisproof} may be done by
integration which is linear, so \eqref{masterthm} holds. Each term is a rational
function of $a$ and $b,$ thus \eqref{masterthm} is true in general.
 \end{proof}

Consider the type $R_I$ orthogonality for a fixed $n\ge 1$
$$
\LL_{a,b}\left(x^k \frac{P_n(x)}{(1+x)^n}\right)=0, \quad 0\le k\le n-1.
$$
This is by Proposition~\ref{hatethisproof} 
$$
\LL_{a,b-n}\left(x^k P_n(x)\right)=0, \quad 0\le k\le n-1.
$$
If $a>-1$ and $b-n>-1$, this is the usual orthogonal polynomial orthogonality, so 
$P_n(x)$ must be a multiple of the Jacobi polynomial $P_n^{(a,b-n)}(x).$
Note also that $P_n^{(a,b-n)}(-1)\neq 0$, which is required.

Once we know an explicit formula for the type $R_I$ polynomials, we can 
find their three term recurrence by considering the higher term coefficients.

To summarize, all we need to do is find the shift 
in parameters when $d_n(x)$ is glued onto a known weight $w(x).$ Then apply this 
shift to the parameters in the usual orthogonal polynomials.  
This is carried out on the next 7 subsections. The 
degeneracy condition $P_n(-\lambda_n/a_n)\neq 0$ holds because the 
polynomials may be explicitly evaluated at these values.

\subsection{Jacobi polynomials on $[-1,1]$}
\aftersubsection

Recall that the Jacobi polynomials are
\[
P_n^{(a,b)}(x)=\frac{(a+1)_n}{n!} \Hyper21{-n,n+a+b+1}{a+1}{\frac{1-x}{2}}.
\]
In Section~\ref{sec:explicit-type-r_i} we took $d_n(x)=(1+x)^n.$ Here we record
the case $d_n(x)=(1-x)^n$, which shifts $a$ to $a-n$. The values of the linear
functional are given by a beta function evaluation.

\begin{defn}
Let $\LL_{a,b}$ be the linear functional on $V$ such that
\begin{equation}
\LL_{a,b}\left(\frac{(1+x)^k}{(1-x)^n}\right)=2^{k-n} \frac{(b+1)_k}{(a-n+1)_n (a+b+2)_{k-n}}, \quad k,n\ge 0.
\end{equation}
\end{defn}

\begin{thm} Up to a constant $c_n,$ the type $R_I$ polynomials for 
$w(x)=(1-x)^a(1+x)^b$ on $[-1,1]$ and  $d_n(x)=(1-x)^n$
are shifted Jacobi polynomials
$$
p_n(x)= c_n P_n^{(a-n,b)}(x).
$$
\end{thm}

\begin{prop} We have for the monic type $R_I$ polynomials, $\hat p_n(x)$,
$$
\hat p_{n+1}(x)=(x-b_n)\hat p_n(x) -(1-x) \lambda_n \hat p_{n-1}(x)
$$
where
$$
b_n= \frac{b-a+3n+1}{a+b+n+1}, \quad \lambda_n=\frac{2n(n+b)}{(a+b+n)(a+b+n+1)}.
$$
\end{prop}

\begin{remark} 
\label{rem1}
The moments are the same as the moments for the usual 
Jacobi polynomials $P_n^{(a,b)}(x)$ on $[-1,1],$ 
$$
\LL_{a,b}( x^n)=\sum_{s=0}^n \binom{n}{s} (-2)^s  \frac{(a+1)_s}{(a+b+2)_s}.
$$
The usual three-term recurrence coefficients for the monic Jacobi polynomials are
\begin{align*}
B_n&= \frac{(b^2 - a^2)}{(2 n + a + b)(2 n + a + b + 2)},\\ 
\Lambda_n  &=\frac{4 n(n + a)(n + b)(n + a + b)}{
(2 n + a + b - 1)(2 n + a + b)^2(2 n + a + b + 1)}
\end{align*}
and Theorem~\ref{weirdCFthm} holds.
\end{remark}

\begin{remark} 
\label{rem2}
If $a=b=1/2$,  then the Catalan numbers $C_{k}=\frac{1}{k+1}\binom{2k}k$ are moments
$$
\{4^k \LL_{1/2,1/2}( x^{2k})= C_{k}: k\ge0\}=\{1,1,2,5,14,42,\dots \}.
$$
If $a=b=-1/2$,  then the central binomial coefficients are moments
$$
\left\{4^k \LL_{-1/2,-1/2}( x^{2k})= \binom{2k}{k}:k\ge0\right\} =\{1,2,6,20, 70,\dots \}.
$$
\end{remark}

\subsubsection{A mixed Jacobi formula}
\aftersubsection

One may alternate inserting $(1-x)$ with $(1+x)$ into the denominator by taking
$$
d_n(x)= (1-x)^{\lceil n/2\rceil}(1+x)^{\lfloor n/2\rfloor}.
$$

\begin{defn} Let $\LL_{a,b}$ be the linear functional on $V$ such that
$$
\LL_{a,b}\left(\frac{(1-x)^k}{(1-x)^{\lceil n/2\rceil}(1+x)^{\lfloor n/2\rfloor}}\right)=
2^{k-n} \frac{(a+1)_{k-\lceil n/2 \rceil}}{(b-\lfloor n/2 \rfloor+1)_{\lfloor n/2 \rfloor} 
(a+b+2)_{k-n}}, \quad k,n\ge 0.
$$
\end{defn}

\begin{thm} Up to a constant, the type $R_I$ polynomials for 
$w(x)=(1-x)^a(1+x)^b$ on $[-1,1]$ and  
$d_n(x)=(1-x)^{\lceil n/2\rceil}(1+x)^{\lfloor n/2\rfloor}$
are shifted Jacobi polynomials
\[
p_n(x)= c_n P_n^{(a-\lceil n/2\rceil,b-\lfloor n/2\rfloor)}(x)
=c_n' \Hyper21{-n,a+b+1}{a-\lceil n/2\rceil+1}{\frac{1-x}{2}}.
\]
\end{thm}

\begin{prop} For the monic type $R_I$ polynomials $\hat p_n(x)$, we have
$$
\hat p_{n+1}(x)=(x-b_n)\hat p_n(x) -(1-(-1)^{n-1}x) \lambda_n \hat p_{n-1}(x),
$$
where
$$
b_n= \frac{1}{a+b+n+1}
\begin{cases}
\times (b-a+1), {\text{ if $n$ is even}}\\ 
\times (b-a), {\text{ if $n$ is odd,}}\\
\end{cases}  
$$
and
$$
\lambda_n= \frac{2n}{(a+b+n)(a+b+n+1)}
\begin{cases}
\times (a+n/2), {\text{ if $n$ is even}}\\ 
\times (b+(n+1)/2), {\text{ if $n$ is odd.}}\\
\end{cases}
$$
\end{prop}

The moments are again given
by Remarks~\ref{rem1} and \ref{rem2}.

\subsection{Jacobi polynomials on $[0,1]$}
\aftersubsection

The Jacobi polynomials on $[0,1]$ are 
\[
P_n^{(a,b)}(1-2x)=\frac{(a+1)_n}{n!}  \Hyper21{-n,n+a+b+1}{a+1}{x}
\]
and have the weight function $w(x)=x^a(1-x)^b$ given by the linear functional  
$$
\MM_{a,b}(f(x))=\frac{\Gamma(a+b+2)}{\Gamma(a+1)\Gamma(b+1)}
\int_0^1 x^a(1-x)^b f(x) dx.
$$

\subsubsection{$d_n(x)=(1-x)^n$}
\aftersubsection

Let's choose $d_n(x)=(1-x)^n$ which naturally glues onto the weight $w(x)=x^a(1-x)^b$. 
So we see that the modified weight function $w'(x)$ 
occurs by replacing $b$ by $b-n$ in $w(x).$  

As before this beta integral may be evaluated.

\begin{defn} 
Let $\MM_{a,b}$ be the linear functional on $V$ such that
$$
\MM_{a,b}\left(\frac{x^k}{(1-x)^n}\right)= \frac{(a+1)_k}{(b-n+1)_n (a+b+2)_{k-n}}, \quad k,n\ge 0.
$$
\end{defn}

\begin{thm} Up to a constant $c_n,$ the type $R_I$ polynomials 
for  $w(x)=x^a(1-x)^b$ on $[0,1]$ and  $d_n(x)=(1-x)^n$
are shifted Jacobi polynomials
\[
p_n(x)= c_n P_n^{(a,b-n)}(1-2x) =c_n' \Hyper21{-n,a+b+1}{a+1}{x}.
\]
\end{thm}

\begin{prop} We have for the monic type $R_I$ polynomials, $\hat p_n(x)$,
$$
\hat p_{n+1}(x)=(x-b_n)\hat p_n(x) -(1-x) \lambda_n \hat p_{n-1}(x)
$$
where
$$
b_n= \frac{a+2n+1}{a+b+n+1}, \quad \lambda_n=\frac{n(n+a)}{(a+b+n)(a+b+n+1)}.
$$
\end{prop}

\begin{remark} The moments are the same as the moments for the usual 
Jacobi polynomials $P_n^{(a,b)}(x),$ on $[0,1]$ instead of $[-1,1],$ 
evaluated by beta functions,
$$
\MM_{a,b}( x^k)= \frac{(a+1)_k}{(a+b+2)_k}.
$$
\end{remark}

\begin{remark} If $a=b=1/2$,  then the Catalan numbers are moments
$$
\{4^k\MM_{1/2,1/2}( x^k)= C_{k+1}:k\ge0\} =\{1,2,5,14,42,\dots \}.
$$
\end{remark}

\subsubsection{$d_n(x)=x^n$}
\aftersubsection

Let's choose $d_n(x)=x^n$ which naturally glues onto the weight $w(x)=x^a(1-x)^b$. 
So we see that the modified weight function $w'(x)$ 
occurs by replacing $a$ by $a-n$ in $w(x).$  

As before this beta integral may be evaluated.

\begin{defn} 
  Let $\MM_{a,b}$ be the linear functional on $V$ such that
  $$
  \MM_{a,b}\left(\frac{(1-x)^k}{x^n}\right)= \frac{(b+1)_k}{(a-n+1)_n (a+b+2)_{k-n}}, \quad k,n\ge 0.
  $$
\end{defn}

\begin{thm} Up to a constant the type $R_I$ polynomials 
for  $w(x)=x^a(1-x)^b$ on $[0,1]$ and  $d_n(x)= x^n$
are shifted Jacobi polynomials
\[
p_n(x)= c_n P_n^{(a-n,b)}(1-2x) =c_n' \Hyper21{-n,a+b+1}{a-n+1}{x}.
\]
\end{thm}

\begin{prop} We have for the monic type $R_I$ polynomials, $\hat p_n(x)$,
$$
\hat p_{n+1}(x)=(x-b_n)\hat p_n(x) -xa_n \hat p_{n-1}(x)
$$
where
$$
b_n=  \frac{a-n}{a+b+n+1}, \quad a_n=\frac{n(b+n)}{(a+b+n)(a+b+n+1)}.
$$
\end{prop}

\subsection{Laguerre polynomials}
\label{sec:laguerre-polynomials}
\aftersubsection

The Laguerre polynomials
\[
L_n^{a}(x)=\frac{(a+1)_n}{n!} \Hyper11{-n}{a+1}{x}
\]
have the linear functional
$$
\LL_a(f(x))=\frac{1}{\Gamma(a+1)}
\int_0^\infty x^a e^{-x} f(x) dx.
$$
Thus the choice of $d_n(x)=x^n$ shifts $a$.

\begin{defn} 
Extend the linear functional $\LL_{a}$ to $V$ by
$$
\LL_{a}\left(\frac{1}{d_n(x)}\right)= \frac{1}{(a-n+1)_n },  \quad n\ge1.
$$
\end{defn}

\begin{thm} Up to a constant the type $R_I$ polynomials for 
$w(x)=x^ae^{-x}$ on $[0,\infty)$ and  
$d_n(x)=x^n$
are shifted Laguerre polynomials
\[
p_n(x)= c_n L_n^{a-n}(x) =c_n' \Hyper11{-n}{a-n+1}{x}.
\]
\end{thm}

\begin{prop} We have for the monic type $R_I$ polynomials, $\hat p_n(x)$,
$$
\hat p_{n+1}(x)=(x-b_n)\hat p_n(x) -x a_n \hat p_{n-1}(x)
$$
where
$$
b_n= a-n, \quad a_n=n.
$$
\end{prop}

\begin{remark}
The monic Laguerre polynomials have 
$$
B_n=2n+a+1, \quad \Lambda_n=n(n+a).
$$
Both the Laguerre and the type \( R_I \) Laguerre have moments
$$
\LL_a( x^{k})= (a+1)_k.
$$
\end{remark}

\subsection{Meixner polynomials}
\label{sec:meixner-polynomials}
\aftersubsection

The Meixner polynomials
\[
M_n(x;b,c)= \Hyper21{-n,-x}{b}{1-\frac{1}{c}}
\]
have for their linear functional
$$
\LL_{b,c}(f(x))=(1-c)^{b}
\sum_{x=0}^\infty \frac{(b)_x}{x!} c^x f(x).
$$
Thus the choice of $d_n(x)=(x+b-1)(x+b-2)\cdots (x+b-n)$ shifts $b$
to $b-n$.

\begin{defn} 
Extend the linear functional $\LL_{b,c}$ to $V$ by
$$
\LL_{b,c}\left(\frac{1}{d_n(x)}\right)= \frac{(1-c)^n}{(b-n)_n },  \quad n\ge 1.
$$
\end{defn}

\begin{thm} Up to a constant, the type $R_I$ polynomials for 
$\LL_{b,c}$ and  $d_n(x)=(x+b-n)_n$
are shifted Meixner polynomials
\[
p_n(x)= c_nM_n(x;b-n,c) = c_n \Hyper21{-n, -x}{b-n}{1-\frac{1}{c}}.
\]
\end{thm}

\begin{prop} We have for the monic type $R_I$ polynomials, $\hat p_n(x)$,
  $$
  \hat p_{n+1}(x)=(x-b_n)\hat p_n(x) -(a_nx+\lambda_n)  \hat p_{n-1}(x),
  $$
where
$$
b_n= \frac{n-(2n+1)c+bc}{1-c}, \quad a_n= \frac{cn}{1-c},
\quad \lambda_n = \frac{cn(b-n)}{1-c}.
$$
\end{prop}

\begin{remark} 
  The monic Meixner polynomials have 
  $$
  B_n=\frac{n+(n+b)c}{1-c}, \quad \Lambda_n=\frac{n(n+b-1)c}{(1-c)^2}.
  $$
The  moments of each of these Meixner and type \( R_I \) Meixner are 
$$
\LL_{b,c}( x^{k})= \sum_{j=1}^k S(k,j) (b)_j \left( \frac{c}{1-c}\right)^j,
$$
where $S(k,j)$ are the Stirling numbers of the second kind.
\end{remark}

\subsection{Little $q$-Jacobi polynomials}
\aftersubsection

The little $q$-Jacobi polynomials are
\[
p_n(x;a,b|q)= \qHyper21{q^{-n}, abq^{n+1}}{aq}{q;qx}.
\]

The linear functional for the little \( q \)-Jacobi orthogonality is
\begin{equation}
\label{littleq}
\LL_{a,b}(f(x))=\frac{(aq)_\infty}{(abq^2)_\infty}
\sum_{x=0}^\infty \frac{(bq)_x}{(q)_x} (aq)^x f(q^x).
\end{equation}

Choosing $d_n(x)=(bx;q^{-1})_n$ we see that  $b$ shifts to $bq^{-n}$ as
$$
\frac{(bq)_x}{d_n(q^x)}=\frac{(bq^{1-n})_x}{(bq^{1-n})_n},
$$
so the next theorem results using the extension
$$
\LL_{a,b}\left(\frac{1}{d_n(x)}\right)= \frac{(abq^{2-n})_n}{(bq^{1-n})_n}, \quad n\ge 1.
$$

\begin{thm} Up to a constant, the type $R_I$ polynomials for 
the little $q$-Jacobi polynomials with $\LL_{a,b}$ given by \eqref{littleq} 
and  $d_n(x)=(bx;q^{-1})_n$
are shifted little $q$-Jacobi polynomials
$$
p_n(x)= c_n p_n(x;a,bq^{-n}|q).
$$
\end{thm}

\begin{prop} We have for the type $R_I$ monic polynomials, $\hat p_n(x)$,
  $$
  \hat p_{n+1}(x)=(x-b_n)\hat p_n(x) -(a_nx+\lambda_n)  \hat p_{n-1}(x),
  $$
where
$$
b_n= q^n\frac{1+a-aq^n-aq^{n+1}}{1-abq^{n+1}}, 
\quad 
\lambda_n= \frac{aq^{2n-1}(1-q^n)(1-aq^n)}{(1-abq^n)(1-abq^{n+1})},
\quad a_n = -\frac{abq^{n}(1-q^n)(1-aq^n)}{(1-abq^n)(1-abq^{n+1})}.
$$
\end{prop}

\begin{remark} The  moments are 
$$
\LL_{a,b}( x^{k})=\frac{(aq;q)_k}{ (abq^{2};q)_k}.
$$
\end{remark}

\subsection{Big $q$-Jacobi polynomials}
\aftersubsection

The big $q$-Jacobi polynomials are
\[
P_n(x;a,b,c;q)= \qHyper32{q^{-n},abq^{n+1},x}{aq,cq}{q;q}.
\]

The linear functional for orthogonality is given by a $q$-integral
\begin{equation}
\label{bigq}
\LL_{a,b,c}(f(x))=\frac{1}{aq(1-q)}
\frac{(aq,bq,cq,abq/c;q)_\infty}{(q,abq^2,c/a,aq/c;q)_\infty}
\int_{cq}^{aq} \frac{(x/a,x/c;q)_\infty}{(x,bx/c;q)_\infty} f(x) d_q(x).
\end{equation}

There are two choices for $d_n(x)$ which shift parameters 
\begin{align*}
d_n(x) &=(bx/cq;q^{-1})_n, \ b\rightarrow bq^{-n},\\
d_n(x) &=(x/a;q)_n, \ a\rightarrow aq^{-n}. 
\end{align*} 

Extending $\LL_{a,b,c}$ may be accomplished via
$$
\LL_{a,b}\left(\frac{1}{d_n(x)}\right)=
\begin{cases}
\frac{(abq^{2-n})_n}{(bq^{1-n})_n (abq^{1-n}/c)_n}, {\text{ if }} d_n(x)=(bx/cq;q^{-1})_n, n\ge 1\\
\frac{(abq^{2-n})_n(aq^{1-n}/c)_n}{(aq^{1-n})_n (abq^{1-n}/c)_n (c/a)_n}, {\text{ if }} d_n(x)=(x/a;q)_n, n\ge 1.
\end{cases}
$$

\begin{thm} Up to a constant, the type $R_I$ polynomials for 
the big $q$-Jacobi polynomials with $\LL_{a,b,c}$ given by \eqref{bigq} 
and  $d_n(x)=(bx/cq;q^{-1})_n$
are shifted big $q$-Jacobi polynomials
$$
p_n(x)= c_n p_n(x;a,bq^{-n};q).
$$
Also choosing $d_n(x)=(x/a;q)_n$
we obtain the shifted big $q$-Jacobi polynomials
$$
p_n(x)= c_n p_n(x;aq^{-n},b;q).
$$
\end{thm}

\begin{prop} We have for the type $R_I$ monic polynomials, $\hat p_n(x)$,
$$
\hat p_{n+1}(x)=(x-b_n)\hat p_n(x) -(1-bxq^{-n}/c) \lambda_n \hat p_{n-1}(x),
\quad d_n(x)=(bx/cq;q^{-1})_n,
$$
where  
\begin{align*}
b_n&= -q \frac{ab-aq^n-cq^n-acq^n+acq^{2n}+acq^{2n+1}}
{1-abq^{n+1}} , \\ 
\lambda_n&=   -acq^{n+1} \frac{(1-q^n)(1-aq^n)(1-cq^n)}
{(1-abq^{n})(1-abq^{n+1})}.
\end{align*}
\end{prop}

\begin{prop} We have for the type $R_I$ monic polynomials, $\hat p_n(x)$,
$$
\hat p_{n+1}(x)=(x-b_n)\hat p_n(x) -(1-xq^{n-1}/a) \lambda_n \hat p_{n-1}(x),
\quad d_n(x)=(x/a;q)_n,
$$
where  
\begin{align*}
b_n&= q^{-n} \frac{a+aq-aq^{n+1}-abq^{n+1}-acq^{n+1}+cq^{2n+1}}
{1-abq^{n+1}} , \\ 
\lambda_n&=   a^2q^{2-2n} \frac{(1-q^n)(1-bq^n)(1-cq^n)}
{(1-abq^{n})(1-abq^{n+1})}.
\end{align*}
\end{prop}

\subsection{The Askey--Wilson polynomials}
\aftersubsection

Here we consider separately the absolutely continuous case and the purely discrete case.

\subsubsection{The continuous case}
\aftersubsection

The Askey--Wilson polynomials are defined by 
\[
p_n(x;a,b,c,d|q)= \frac{(ab,ac,ad;q)_n}{a^n}  \qHyper43{q^{-n}, abcdq^{n-1},az,a/z}{ab,ac,ad}{q;q},
\]
$$
z=e^{i\theta}, \quad x=\cos\theta=(z+1/z)/2.
$$

Note that
$$
(Az,A/z;q)_n=\prod_{j=0}^{n-1} (1-2Axq^j+A^2q^{2j})
$$
is a polynomial in $x$ of degree $n$. Thus $p_n(x;a,b,c,d|q)$ is a function of $x$.

The weight function for the Askey--Wilson polynomials is
$$
\LL_{a,b,c,d}(r(x))=\frac{(q,ab,ac,ad,bc,bd,cd)_\infty}{2\pi(abcd)_\infty}
\int_0^\pi r((e^{i\theta}+e^{-i\theta})/2) w(\theta,a,b,c,d) d\theta,
\quad \LL_{a,b,c,d}(1)=1,
$$
where
$$
w(\theta,a,b,c,d)=\frac{(e^{2i\theta},e^{-2i\theta})_\infty}
{(ae^{i\theta},ae^{-i\theta},be^{i\theta},be^{-i\theta},
ce^{i\theta},ce^{-i\theta},de^{i\theta},de^{-i\theta})_\infty}.
$$

Let 
$$
d_n(x)=(bz/q,b/zq;q^{-1})_n= \prod_{j=0}^{n-1} (1-2bxq^{-1-j}+b^2q^{-2-2j}).
$$
We next define the extension of $\LL_{a,b,c,d}$ to $V$ using the 
Askey--Wilson integral and 
$$
\frac{w(\theta,a,b,c,d)}{d_n(x)}= w(\theta,a,bq^{-n},c,d).
$$

\begin{defn}
Suppose $\LL_{a,b,c,d}$ is a linear functional on $V$  such that
\begin{equation}
\label{myL}
\LL_{a,b,c,d}\left(\frac{(cz,c/z;q)_j (az,a/z;q)_k}{(bzq^{-n},bq^{-n}/z;q)_n}\right)
=
\frac{(cd;q)_j (ac;q)_{k+j} (ad;q)_k}
{(abq^{k-n};q)_{n-k}(bcq^{j-n};q)_{n-j}(bdq^{-n};q)_n(abcd;q)_{k+j-n}}.
\end{equation}
\end{defn}

\begin{thm}
The type $R_I$ polynomials for $\LL_{a,b,c,d}$ and denominator polynomials 
$$
d_n(x)= \prod_{j=0}^{n-1} (1-2bxq^{-1-j}+b^2q^{-2-2j})
$$
are shifted Askey--Wilson polynomials 
\[
p_n(x;a,bq^{-n},c,d|q)= \frac{(abq^{-n},ac,ad;q)_n}{a^n}  \qHyper43{q^{-n},abcd/q,az,a/z}{ac,ad,abq^{-n}}{q;q}.
\]
\end{thm}

Note that
$$
d_n(x)/d_{n-1}(x)=1-2bxq^{-n}+b^2q^{-2n}
$$
which is the factor in the type $R_I$ recurrence relation.

\begin{prop} We have for the monic type $R_I$ polynomials, $\hat p_n(x)$,
$$
\hat p_{n+1}(x)=(x-b_n)\hat p_n(x) - \lambda_n (1-2bxq^{-n}+b^2q^{-2n}) \hat p_{n-1}(x),
$$
where
\begin{align*}
b_n&= \frac{bq^{-n}+q^n/b}2-\frac{q^{1+2n}}{2b(1-abcdq^{n-1})}\bigg(
  (1-abq^{-1-n}) (1-bcq^{-1-n})(1-bdq^{-1-n})\\
& \qquad -(1-q^{-1-n})(1-abcd/q)(1-b^2q^{-1-2n})\bigg),\\
 \lambda_n&= \frac{(1-q^n)(1-acq^{n-1})(1-adq^{n-1})(1-cdq^{n-1})}
 {4(1-abcdq^{n-2})(1-abcdq^{n-1})}.
\end{align*}
\end{prop}

\begin{proof}[Sketch of Proof]  The value $z=z_n=bq^{-n}$ puts $x=x_n=1/2(bq^{-n}+q^n/b)$, 
and then $d_n(x_n)/d_{n-1}(x_n)=0$. For this 
value of $x_n$, $p_n(x_n)$ is evaluable as a product by 
the 1-balanced ${}_3\phi_2$ evaluation.  This choice also allows $p_{n+1}(x_n)$ to be a sum of 
2 terms, using the Sears transformation for a 1-balanced  ${}_4\phi_3.$  This determines the value of 
$b_n$, and $\lambda_n$ can be found by finding the coefficients of $x^{n}.$
\end{proof}

\subsubsection{The $q$-Racah case}
\aftersubsection

For completeness we record the analogous results for the $q$-Racah polynomials.

\begin{defn} For $0\le n\le N$, let $p_n(X;b,c,d,N;q)$ be 
the polynomial of degree $n$ in $X=\mu(x)= q^{-x}+cdq^{x+1}$ 
\[
p_n(\mu(x);b,c,d,N;q)= \qHyper43{q^{-n},bq^{n-N},q^{-x},cdq^{x+1}}{q^{-N},bdq,cq}{q;q}.
\]
\end{defn}

Since
$$
(q^{-x};q)_j(cdq^{x+1};q)_j=\prod_{s=0}^{j-1} (1-q^s\mu(x)+cdq^{1+2s})
$$
$p_n(\mu(x);b,c,d,N;q)$ is a polynomial in $\mu(x)=X$ of degree $n$.

Let 
\[
  d_n(X)=\prod_{j=0}^{n-1} (1-Xq^j/bd+q^{2j+1}c/b^2d), 
\]
or
\[
  d_n(\mu(x))=(q^{-x}/bd, cq^{x+1}/b;q)_n.
\]

\begin{defn} 
Let $\MM_{b,c,d,N}$ be the linear functional defined on $V$ by
\[
\MM_{b,c,d,N}(r(X))= \sum_{x=0}^N vwp(x,b,c,d,N) r(\mu(x)),
\]
where
\[
vwp(x,b,c,d,N) = \frac{(dq)_N (cq/b)_N}{(cdq^2)_N (1/b)_N}
\frac{(cdq)_x}{(q)_x}\frac{1-cdq^{1+2x}}{1-cdq}
\frac{(q^{-N})_x}{(cdq^{N+2})_x} \frac{(cq)_x}{(dq)_x}
\frac{(bdq)_x}{(cq/b)_x} (q^{N}/b)^{x}.
\]
See \cite[(II-21)]{GR}. Here the constants have been chosen, using the very well
poised ${}_6\phi_5$ summation theorem, so that $\MM_{b,c,d,N}(1)=1.$
\end{defn}

The classical orthogonal polynomials for $\MM_{b,c,d,N}$ are $q$-Racah polynomials
$$
\MM_{b,c,d,N}\left(p_n(X;b,c,d,N;q)p_m(X;b,c,d,N;q)\right)=0  \qquad{\text{if }} n\neq m.
$$

Note that gluing does occur as
$$
\frac{vwp(x,b,c,d,N)}{d_n(\mu(x))} = c(n,b,c,d,N)\times  vwp(x,bq^{-n},c,d,N),
$$
for some constant $c(n,b,c,d,N)$ independent of $x$.

\begin{thm} 
\label{typeIracah}
The polynomials $p_n(X;bq^{-n},c,d,N;q)$, $0\le n\le N$, are the type 
$R_I$ polynomials 
for the $q$-Racah linear functional $\MM_{b,c,d,N}$ with 
\begin{align*}
d_n(X)&=\prod_{j=0}^{n-1} (1-Xq^j/bd+q^{2j+1}c/b^2d), \\
d_n(\mu(x))&=(q^{-x}/bd, cq^{x+1}/b;q)_n.
\end{align*}
\end{thm}

\begin{prop} The monic type $R_I$ $q$-Racah polynomials satisfy
for $0\le n \le N-1$
$$
\hat p_{n+1}(X)=(X-b_n)\hat p_n(X) -\lambda_n(1-Xq^{n-1}/bd+q^{2n-1}c/b^2d)\hat p_{n-1}(X)
$$
where
\begin{align*}
b_n&=-( -b+bd(-1+q^{N-n}+q^{N-n+1}-q^{N+1})+q^n\\
& \qquad + c(q^n-q^{2n}-q^{2n+1}+q^{N+n+1})-bcdq^{N+1}+cdq^{N+n+1})/(bq^n-q^{N}),\\
\lambda_n&= dq^{1-2n} \frac{(1-q^n)(1-cq^{n})(1-q^{N-n+1})(1-dq^{N-n+1})}
 {(1-q^{N-n}/b)(1-q^{N-n+1}/b)}.
\end{align*}
\end{prop}

\section{A special type \( R_I \) polynomial}
\label{powers}

In this section we consider the case $b_n=0$ for orthogonal polynomials, so 
that the polynomials are either even or odd. Hermite polynomials are one example. 
We choose $d_n(x)=(1+ax)^n,$ then modify the three term recurrence by inserting the factor 
$1+ax$, independent of $n$, for a type $R_I$ polynomial, see \eqref{modevenodd}. 
The new linear functional $\LL_a$ has new values on the polynomials, so the moments do 
change, unlike the previous examples.

General results for the type $R_I$ polynomials and their moments, in terms of the 
original orthogonal polynomials, are given in Proposition~\ref{genr},
Proposition~\ref{genmom}, and Theorem~\ref{genthm}. We do not know 
a representing measure for the new linear functional $\LL_a$ in terms of an 
original measure, even in the Hermite case.

\subsection{General results}
\aftersubsection

Let $p_n(x)$ be orthogonal polynomials defined by
\begin{equation}
\label{OPrr}
p_{n+1}(x)=x p_n(x) -\lambda_n p_{n-1}(x)
\end{equation}
with nonzero moments $\mu_{2n}=\LL(x^{2n}).$

Consider a type $R_I$ version of $p_n$ defined by
\begin{equation}
\label{modevenodd}
r_{n+1}(x)=x r_n(x) -\lambda_n (1+ax) r_{n-1}(x).
\end{equation}
These polynomials are rescaled versions of the original orthogonal polynomials. 
Proposition~\ref{genr} follows easily by rescaling.

\begin{prop} 
\label{genr}
We have
$$
r_n(x) =p_n\left( \frac{x}{\sqrt{1+ax}}\right) (\sqrt{1+ax})^n.
$$
\end{prop}

Next we see how the moments are related. 
Let $\LL_a$ be the linear functional for these polynomials. 
Let $\mu_n$ be the moments for the orthogonal polynomials in \eqref{OPrr}, and let
$\theta_n=\LL_a(x^n)$ be the moments for the type $R_I$ polynomials \eqref{modevenodd}.

\begin{prop}
\label{genmom}
The moments $\theta_n=\LL_a(x^n)$ for the above type $R_I$ polynomials are given by
\begin{align*}
\theta_{2n} &=\sum_{k\ \mathrm{even}}\binom{n+k/2}{k} a^k\mu_{2n+k},\\
\theta_{2n+1}&=\sum_{k\ \mathrm{odd}}\binom{n+(k+1)/2}{k}  a^k\mu_{2n+1+k}.
\end{align*}
\end{prop}

\begin{proof} We use the combinatorial interpretation of $\theta_n$ as weighted
  Motzkin-Schr\"oder paths and $\mu_{2n}$ as weighted Dyck paths, where a Dyck
  path is a Motzkin path with no horizontal steps.

Take the paths for $\theta_{2n}$, which start at $(0,0)$, end at $(2n,0)$, and
stay at or above the $x$-axis. There are no horizontal steps, as $b_n=0$ in
\eqref{modevenodd}. There are diagonal down steps starting at $y$-coordinate $n$ with
weight $\lambda_n$, and vertical down steps starting at $y$-coordinate $n$ with
weight $a\lambda_n.$ If there are $k$ vertical down steps, where $k$ must be even,
these contribute a weight of $a^k$. We can change these $k$ steps to diagonal down steps
to obtain a weighted Dyck path from $(0,0)$ to $(2n+k,0)$. This is a term in the
combinatorial expansion for $\mu_{2n+k}$. But each such Dyck path for
$\mu_{2n+k}$ occurs $\binom{n+k/2}{k}$ times, by choosing which of the $n+k/2$
diagonal down steps are switched to vertical down steps.
 
The proof for $\theta_{2n+1}$ is basically the same.
\end{proof}

These moments may also be connected via Chebyshev polynomials of the
  second kind
\begin{align*}
U_{2n}(x) &= (-1)^n \Hyper21{-n,n+1}{1/2}{x^2},\\
  U_{2n+1}(x)&=(-1)^n(n+1)x\Hyper21{-n,n+2}{3/2}{x^2}.
\end{align*}

\begin{prop} 
\label{gencor} 
The moments $\theta_n$ of the Chebyshev polynomials \( U_n(x) \) satisfy
$$
\theta_n=\LL_a(x^n)=\LL(x^n w_n(x,a)),
$$
where
\begin{align*}
w_{2n}(x,a) &=\sum_{k\  even}\binom{n+k/2}{k} a^k x^k=
\Hyper21{-n,n+1}{1/2}{-a^2x^2/4}=(-1)^n U_{2n}(aix/2),\\
w_{2n+1}(x,a)&=\sum_{k\  odd}\binom{n+(k+1)/2}{k}  a^k x^k=
(n+1)ax  \Hyper21{-n,n+2}{3/2}{-a^2x^2/4}\\
            &= 2i(-1)^{n+1} U_{2n+1}(aix/2).
\end{align*}
\end{prop}

The moment generating functions, which are given by the continued fractions 
in Theorem~\ref{weirdCFthm} are related.

\begin{thm} 
\label{genthm}
As formal power series in $t$, we have for the usual orthogonal polynomials
$$
\sum_{n=0}^\infty \mu_{2n}t^{2n}=\LL\left(\frac{1}{1-xt}\right)=\LL\left(\frac{1}{1-x^2t^2}\right),
$$
and for the type $R_I$ polynomials 
$$
\sum_{n=0}^\infty \theta_n t^n= \LL_a\left(\frac{1}{1-xt}\right)=\LL\left(  \frac{1}{1-x^2t(a+t)}\right).
$$
\end{thm}

\subsection{Explicit examples}
\label{sec:exex}
\aftersubsection

\subsubsection{Chebyshev polynomials}
\aftersubsection

\begin{prop} If $b_n=b$, $a_n=a$ and $\lambda_n=\lambda$ are constant, 
then the  type $R_I$ monic polynomials are
$$
p_n(x)=U_n\left( \frac{x-b}{2\sqrt{ax+\lambda}}\right)(\sqrt{ax+\lambda})^n
$$
where \( U_n(x) \) is the Chebyshev polynomial.
\end{prop}

Note that these are the polynomials we considered in Theorem~\ref{thm:hankel}.

\subsubsection{Hermite polynomials}
\aftersubsection

\begin{prop} 
\label{RIHerm}
If $b_n=0$, $a_n=an$ and $\lambda_n=n$, 
then the  type $R_I$ monic polynomials are $p_n(x)$
satisfying
$$
\sum_{n=0}^\infty \frac{p_n(x)}{n!}t^n=
e^{xt-(1+ax)t^2/2}.
$$
Also
$$
p_n(x) =\mathrm{He}_n\left( \frac{x}{\sqrt{1+ax}}\right) (\sqrt{1+ax})^n,
$$
where $\mathrm{He}_n(x)$ are the monic Hermite polynomials normalized by
$$
\mathrm{He}_{n+1}(x)=x\mathrm{He}_n(x)-n\mathrm{He}_{n-1}(x), \quad \mathrm{He}_{-1}(x)=0,\quad \mathrm{He}_0(x)=1.
$$
\end{prop}

\section{Combinatorics of three type \( R_I \) polynomials}
\label{sec:combinatorics}

In this section we study combinatorial aspects of some type $R_I$ orthogonal
polynomials considered in Sections~\ref{glue} and \ref{powers}.

\subsection{Type \( R_I \) Hermite polynomials}
\aftersubsection

The classical Hermite  polynomials $\mathrm{He}_n(x)$ are the generating functions for 
matchings (or involutions) on $n$ points. Their moments count perfect matchings on 
$n$ points. In this subsection we give the corresponding results for the 
type $R_I$ Hermite polynomials in Section~\ref{sec:exex}.

Let $H_n(x,a)$ be the type $R_I$ Hermite polynomials given by 
Proposition~\ref{RIHerm}. The following result follows from either Theorem~\ref{thm:favard} 
or the exponential generating function for $H_n(x,a).$

\begin{prop}
  The type \( R_I \) Hermite polynomial $H_n(x,a)$ is the generating function
  for involutions on $n$ points, where 1-cycles are weighted by $x$, and
  2-cycles are two colored, with weights $-1$ and $-ax$.
\end{prop}

The combinatorics of the moments can be given using Proposition~\ref{genmom}.

\begin{prop}
The moments $\theta_m=\LL(x^m)$ for the type $R_I$ Hermite polynomials are the generating 
functions for the following perfect matchings in which the edges are colored red and blue.
\begin{enumerate}
\item If $m=2n$,  the perfect matching is on $2n+2K$ points, 
and $2K$ of these $n+K$ edges are colored red, each of weight $a$, 
and the remaining edges are colored blue, each of weight \( 1 \),
for some $0\le K\le n.$\\
\item If $m=2n+1$,  the perfect matching is on $2n+2K+2$ points, 
and $2K+1$ of these $n+K+1$ edges are colored red, each of weight $a$, 
and the remaining edges are colored blue, each of weight \( 1 \),
for some $0\le K\le n.$\\ 
\end{enumerate}  
\end{prop}

Here is a simple linearization result:
$$
H_n(x,a)H_m(x,a)=\sum_{s=0}^{min(m,n)} \binom{n}{s}\binom{m}{s} s! (1+ax)^s 
H_{n+m-2s}(x,a).
$$

\subsection{Type \( R_I \) Laguerre polynomials}
\aftersubsection

The classical Laguerre polynomials have moments $\mu_n=(a+1)_n=(A)_n$ if 
$A=a+1.$ When $A=1$ this is $\mu_n=n!,$ so the number of weighted 
Motzkin paths is the number of permutations of length $n$. Viennot \cite{ViennotLN} used 
{\it{Laguerre histories}} to give a bijection which explained this fact, and implied 
many weighted versions.

In Section~\ref{sec:laguerre-polynomials} the type $R_I$ Laguerre polynomials are given. 
The moments remain $\mu_n=(a+1)_n=(A)_n$, but we now have weighted 
Motzkin-Schr\"oder paths with different weights, 
$$
b_n=a-n, \quad a_n=n, \quad  \lambda_n=0.
$$
Note that since $\lambda_n=0$, Motzkin-Schr\"oder paths become Schr\"oder paths.
In terms of weighted lattice paths we have the following proposition.

\begin{prop}\label{prop:laguerre}
  For $b_n=a-n$, $a_n=n$, and $\lambda_n=0$, we have
\[
\sum_{\pi\in \Sch_n} \wt (\pi) = (a+1)_n.
\]
\end{prop}

In this subsection we give a combinatorial proof of
Proposition~\ref{prop:laguerre} using a \emph{type $R_I$ Laguerre history}.

Let $\Sch_n'$ denote the set of Schr\"oder paths $\pi\in \Sch_n$ that contain no
peaks $(U,V)$. For $\pi\in \Sch_n'$, define $\wt'(\pi)$ to be the weight of
$\pi$ with respect to $b_n = a+1$, $a_n=n$, and $\lambda_n=0$.

\begin{lem}\label{lem:wt=wt'}
  We have
\[
\sum_{\pi\in \Sch_n} \wt (\pi) =\sum_{\pi\in \Sch'_n} \wt' (\pi).
\]
\end{lem}
\begin{proof}
  For $\pi,\sigma\in \Sch_n$ define a relation $\pi\sim\sigma$ if $\pi$ is
  obtained from $\sigma$ by a sequence of replacing a peak $(U,V)$ by a
  horizontal step $H$ or vice versa. Since a peak $(U,V)$ and a horizontal step
  $H$ have the same starting and ending points, it is easy to see that $\sim$ is
  an equivalence relation on $\Sch_n$ and $\Sch_n'$ is a system of
  representatives of the equivalence classes. Moreover, for each $\pi\in\Sch'_n$,
  we have
\[
\sum_{\sigma\sim \pi} \wt(\sigma) = \wt'(\pi)
\]
because the weight of a peak $(U,V)$ and a horizontal step $H$ starting at
height $n$ are, respectively, $n+1$ and $a-n$, whose sum is $a+1$.
Summing the above equation over $\pi\in\Sch'_n$ gives the desired identity.
\end{proof}

\begin{defn}
  A \emph{type $R_I$ Laguerre history} of length $n$ is a labeled Schr\"oder
  path without peaks \( (U,V) \) from $(0,0)$ to $(n,0)$ in which each vertical down step starting at height
  $h$ is labeled by an integer in $\{1,2,\dots,h\}$. The set of type $R_I$
  Laguerre histories of length $n$ is denoted by $\LH_n$.
\end{defn}

By the definition of a type $R_I$ Laguerre history we have
\begin{equation}
  \label{eq:17}
\sum_{\pi\in \Sch'_n} \wt' (\pi) = \sum_{\pi\in \LH_n} (a+1)^{H(\pi)},
\end{equation}
where $H(\pi)$ is the number of horizontal steps in $\pi$. By
Lemma~\ref{lem:wt=wt'} and \eqref{eq:17}, to show
Proposition~\ref{prop:laguerre} it suffices to show the following proposition.

\begin{prop}\label{prop:LH}
  We have
\[
 \sum_{\pi\in \LH_n} (a+1)^{H(\pi)} = (a+1)_n.
\]
\end{prop}

Recall the well known result \cite[1.3.7~Proposition]{EC1}
\[
  \sum_{\sigma\in\Sym_n} (a+1)^{\cyc(\sigma)} = (a+1)_n,
\]
where $\Sym_n$ is the set of permutations on $[n]$ and \( \cyc(\sigma) \) is the
number of cycles in \( \pi \). To show the above proposition we give a bijection
$\phi:\LH_n\to \Sym_n$ such that \( H(\pi) = \cyc(\phi(\pi))\).

Let $\pi\in\LH_n$. Then $\phi(\pi)$ is the permutation in $\Sym_n$ constructed
as follows. The basic idea is to create a cycle for each horizontal step of
$\pi$ and the vertical down steps following immediately after that.
\begin{itemize}
\item First, consider the leftmost horizontal step. Suppose there are $k$
  vertical down steps, labeled $v_1,\dots,v_k$, following this horizontal step. If
  the horizontal step is between the lines $x=i-1$ and $x=i$, create a cycle
  starting with $i$. Then for $j=1,2,\dots,k$, add at the end of the cycle the
  $v_j$th smallest integer in $[i]$ that have not been used. This creates a
  cycle of length $k+1$ with largest integer $i$.
\item For each of the remaining horizontal steps, from left to right, repeat the
  above process.
\end{itemize}

For example, let $\pi$ be the type $R_I$ Laguerre history in
Figure~\ref{fig:lag}. Then the corresponding permutation $\phi(\pi)$, in cycle
notation, is given by
\[
\phi(\pi) = (4,2,3) (8) (9,7,1) (10) (12) (13,5,11,6).
\]

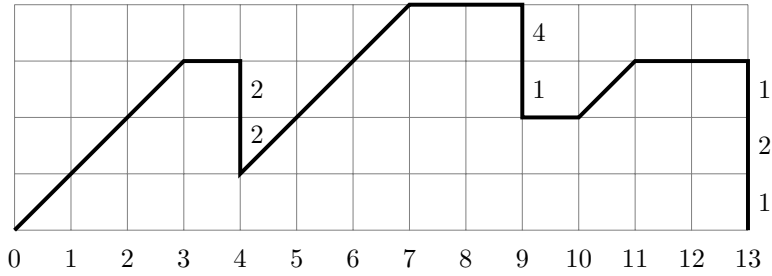
\begin{figure}
  \centering
\begin{tikzpicture}[scale=0.75]
    \draw[help lines] (0,0) grid (13,4);
    \foreach \x in {0,...,13} \draw node at (\x,-.5) {$\x$};
    \draw[line width = 1.5pt] (0,0) -- (3,3) --++ (1,0) --++ (0,-2) --++(3,3) --++ (2,0) --++(0,-2)
    --++(1,0) --++(1,1) --++(2,0) --++(0,-3);
    \node at (4.3,2.5) {$2$};
    \node at (4.3,1.7) {$2$};
    \node at (9.3,3.5) {$4$};
    \node at (9.3,2.5) {$1$};
    \node at (13.3,2.5) {$1$};
    \node at (13.3,1.5) {$2$};
    \node at (13.3,.5) {$1$};
\end{tikzpicture}
  \caption{A type $R_I$ Laguerre history of length $13$.}
  \label{fig:lag}
\end{figure}

It is easy to check that the map $\phi$ is a bijection such that if
$\phi(\pi)=w$, then $H(\pi)$ is equal to the number of cycles in $w$. This proves
Proposition~\ref{prop:LH}.

\subsection{Type \( R_I \) Meixner polynomials}
\aftersubsection

For the Meixner polynomials, the situation is similar to the 
Laguerre polynomials. The moments for the classical Meixner polynomials are
\begin{equation}\label{eq:mei_mo}
\mu_n = \sum_{j=1}^n S(n,j) (b)_j \left( \frac{c}{1-c}\right)^j,
\end{equation}
which involves set partitions and permutations as the fundamental combinatorial objects.

The type $R_I$ Meixner polynomials in Section~\ref{sec:meixner-polynomials} retain 
these moments, but with different paths and weights. In this section we develop a
\emph{type $R_I$ Meixner history} to prove \eqref{eq:mei_mo}. 
To simplify matters we reformulate the formula for the moments. Let
$d=c/(1-c)$, so that
\begin{equation}\label{eq:bala3}
b_n = n-dn+bd-d, \qquad a_n = nd, \qquad \lambda_n = bdn-dn^2.
\end{equation}

\begin{prop}\label{prop:meixner}
  Let $b_n,a_n$, and $\lambda_n$ be given by \eqref{eq:bala3}. Then
\[
\mu_n = \sum_{j=1}^n S(n,j) (b)_j d^j.
\]
\end{prop}

In this subsection we give a combinatorial proof of
Proposition~\ref{prop:meixner}. 

For $\pi\in\MS_n$ define $\wt(\pi)$ to be the weight of $\pi$ with
respect to the weights in Proposition~\ref{prop:meixner}. Then a combinatorial restatement of 
Proposition~\ref{prop:meixner} is Proposition~\ref{prop:meixner2}.

\begin{prop}\label{prop:meixner2}
  We have
\[
\sum_{\pi\in \MS_n} \wt (\pi)= \sum_{j=1}^n S(n,j) (b)_j d^j.
\]
\end{prop}

Let $\MS_n'$ denote the set of Motzkin-Schr\"oder paths $\pi\in \MS_n$ that
contain no peaks $(U,V)$. For $\pi\in \MS_n'$, define $\wt'(\pi)$ to be the
weight of $\pi$ with respect to 
\[
b_n = n+bd, \qquad a_n=nd, \qquad \lambda_n=bdn-dn^2.
\]
Then the following lemma is proved by the same argument as in the proof of
Lemma~\ref{lem:wt=wt'}.

\begin{lem}\label{lem:wt=wt'2}
  We have
\[
\sum_{\pi\in \MS_n} \wt (\pi) =\sum_{\pi\in \MS'_n} \wt' (\pi).
\]
\end{lem}

Let $\MS_n''$ denote the set of Motzkin-Schr\"oder paths $\pi\in \MS'_n$ that
contain no diagonal down steps. Note that $\MS_n''\subset\Sch_n$. For $\pi\in \MS_n''$,
define $\wt''(\pi)$ to be the product of the weight of each step, where the
weight of a step starting at height $n$ is given by
\begin{itemize}
\item $1$ if the step is an up step,
\item $nd$ if the step is a vertical down step,
\item $bd+n$ if the step is a horizontal step not followed by a vertical down step, and
\item $bd+b$ if the step is a horizontal step followed by a vertical down step.
\end{itemize}

By defining a relation $\pi \sim \sigma$ if $\pi$ is obtained from $\sigma$ by a
sequence of replacing a pair $(H,V)$ with a diagonal down step $D$ or vice versa, we
similarly obtain the following lemma.

\begin{lem}\label{lem:wt=wt'3}
  We have
\[
\sum_{\pi\in \MS'_n} \wt' (\pi) =\sum_{\pi\in \MS''_n} \wt'' (\pi).
\]
\end{lem}

\begin{defn}
A \emph{type $R_I$ Meixner history} of length $n$ is a path $\pi\in\MS''_n$
together with a labeling such that 
\begin{itemize}
\item each vertical down step starting at height $h$ is labeled by an integer in
  $\{1,2,\dots,h\}$,
\item each horizontal step followed by a vertical down step is not labeled or is
  labeled by $0$, and
\item each horizontal step not followed by a vertical down step and starting at
  height $h$ is not labeled or is labeled by an integer in $\{1,2,\dots,h\}$.
\end{itemize}
\end{defn}

Let $\MH_n$ denote the set of type $R_I$ Meixner histories of length $n$.

For $\pi\in\MH_n$, define $\wt(\pi)$ to be the product of the weight of each
step defined as follows:
\begin{itemize}
\item An up step has weight $1$.
\item A vertical down step has weight $d$.
\item A non-labeled horizontal step has weight $bd$.
\item A horizontal step labeled $0$ has weight $b$.
\item A horizontal step labeled $i$, for $i\ge1$, has weight $1$.
\end{itemize}

By definition it is clear that
\begin{equation}
  \label{eq:MH}
 \sum_{\pi\in \MS''_n} \wt'' (\pi) = \sum_{\pi\in \MH_n} \wt(\pi).
\end{equation}

By Lemmas~\ref{lem:wt=wt'2} and \ref{lem:wt=wt'3} and \eqref{eq:MH},
Proposition~\ref{prop:meixner2} is equivalent to
\begin{equation}\label{eq:Sbd}
 \sum_{\pi\in \MH_n} \wt(\pi)= \sum_{j=1}^n S(n,j) (b)_j d^j.
\end{equation}

Let $S_n$ denote the set of pairs $(P,\sigma)$ of a set partition
$P=\{B_1,\dots,B_k\}$ of $[n]$ and a permutation $\sigma$ of the blocks
$B_1,\dots,B_k$ of $P$. For $(P,\sigma)\in S_n$, define
\[
\wt(P,\sigma) = b^{\cyc(\sigma)} d^{|P|},
\]
where $\cyc(\sigma)$ is the number of cycles in $\sigma$ and $|P|$ is the number
of blocks in $P$.
Since $\sum_{\pi\in \Sym_n} b^{\cyc(\pi)} = (b)_n$, we have
\begin{equation}\label{eq:P,sigma}
\sum_{j=1}^n S(n,j) (b)_j d^j = \sum_{(P,\sigma)\in S_n} \wt(P,\sigma).
\end{equation}

Using \eqref{eq:P,sigma} we can rewrite \eqref{eq:Sbd} as
\begin{equation}
  \label{eq:wt=wt}
 \sum_{\pi\in \MH_n} \wt(\pi)= \sum_{(P,\sigma)\in S_n} \wt(P,\sigma) .
\end{equation}
To prove \eqref{eq:wt=wt} we find a weight-preserving bijection $\psi:\MH_n\to
S_n$.

Let $\pi\in\MH_n$. Then $\pi$ has exactly $n$ steps which are not vertical down steps, say
$A_1,\dots,A_n$ from left to right. Note that the ending point of $A_i$ has
$x$-coordinate $i$. We create \emph{available blocks} and \emph{cycles of
  blocks} as follows. We will use the convention that once an available block is
used as an element of a cycle the block is no longer available. Moreover, if
there are several available blocks, these are ordered by their smallest
elements.
\begin{itemize}
\item Initially there are no available blocks and no cycles of blocks.
\item For $i=1,2,\dots,n$, do the following procedure:
  \begin{description}
  \item[Case 1] $A_i$ is an up step. In this case we create a new available block containing a single element $i$.
  \item[Case 2] $A_i$ is a horizontal step not followed by a vertical down step.
    \begin{description}
    \item[Case 2-a] $A_i$ is non-labeled. In this case create a new available block
      containing a single element $i$ and make a cycle consisting only of this
      block.
    \item[Case 2-b] $A_i$ is labeled by $j\ge1$. In this case insert the integer $i$
      in the $j$-th available block.
    \end{description}
  \item[Case 3] $A_i$ is a horizontal step followed by a vertical down step.
    \begin{description}
    \item[Case 3-a] $A_i$ is non-labeled. In this case create a new
      available block containing a single element $i$ and make a cycle starting
      with this block. Suppose that $A_i$ is followed by $k$ vertical down steps
      labeled $r_1,r_2,\dots,r_k$. Then add the $r_1$-th available block at the
      end of the cycle, add the $r_2$-th available block at the end of the
      cycle, and so on. After this process we obtain a cycle consisting of $k+1$
      blocks.
    \item[Case 3-b] $A_i$ is labeled by $0$. Suppose that $A_i$ is followed by
      $k$ vertical down steps labeled $r_1,r_2,\dots,r_k$. First insert $i$ in
      the $r_1$-th available block and then create a new cycle starting with
      this block. Then, as in the previous case, add the $r_2$-th available
      block at the end of the cycle, add the $r_3$-th available block at the end
      of the cycle, and so on. After this process we obtain a cycle consisting
      of $k$ blocks.
    \end{description}
  \end{description}
\end{itemize}

For example, if $\pi$ is the Meixner history in Figure~\ref{fig:MH},
then the cycles of $\psi(\pi)$ are created as in Table \ref{tab:MH} and we get
\[
\psi(\pi) = (\{3,4\}, \{1\}) (\{7\}) (\{8\}, \{5,6\})
(\{12\}, \{11\}, \{9\}, \{10\})
(\{13,14\}, \{2\}).
\]

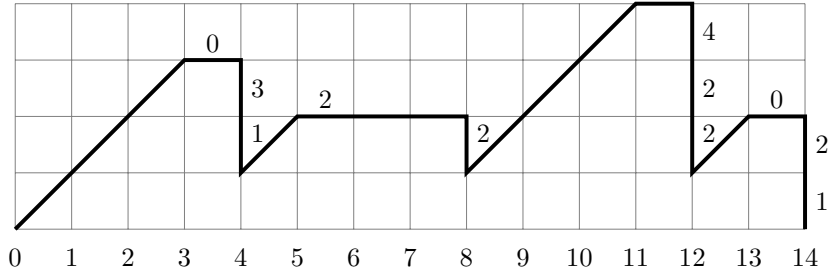
\begin{figure}
  \centering
 \begin{tikzpicture}[scale=0.75]
    \draw[help lines] (0,0) grid (14,4);
    \foreach \x in {0,...,14} \draw node at (\x,-.5) {$\x$};
    \draw[line width = 1.5pt] (0,0) -- (3,3) --++ (1,0) --++ (0,-2) --++(1,1) --++(3,0)
    --++(0,-1) --++(3,3) --++(1,0) --++(0,-3) --++(1,1) --++(1,0) --++(0,-2);
    \node at (3.5,3.3) {$0$};
    \node at (4.3,2.5) {$3$};
    \node at (4.3,1.7) {$1$};
    \node at (5.5,2.3) {$2$};
    \node at (8.3,1.7) {$2$};
    \node at (12.3,3.5) {$4$};
    \node at (12.3,2.5) {$2$};
    \node at (12.3,1.7) {$2$};
    \node at (13.5,2.3) {$0$};
    \node at (14.3,1.5) {$2$};
    \node at (14.3,.5) {$1$};
\end{tikzpicture}
  \caption{A type $R_I$ Meixner history of length $14$.}
  \label{fig:MH}
\end{figure}

\begin{table}
  \centering
\begin{tabular}{| l | l | l |} \hline
  steps except vertical down steps & available blocks & new cycles of blocks \\ \hline
  $A_1$ & $\{1\}$ & \\
  $A_2$ & $\{1\}, \{2\}$ & \\
  $A_3$ & $\{1\}, \{2\}, \{3\}$ & \\
  $A_4$ & $\{1\}, \{2\}, \{3\}$ & $(\{3,4\}, \{1\})$\\
  $A_5$ & $\{2\}, \{5\}$ & \\
  $A_6$ & $\{2\}, \{5,6\}$ & \\
  $A_7$ & $\{2\}, \{5,6\}$ & $(\{7\})$\\
  $A_8$ & $\{2\}, \{5,6\}$ & $(\{8\}, \{5,6\})$\\
  $A_9$ & $\{2\}, \{9\}$ & \\
  $A_{10}$ & $\{2\}, \{9\}, \{10\}$ & \\
  $A_{11}$ & $\{2\}, \{9\}, \{10\}, \{11\}$ & \\
  $A_{12}$ & $\{2\}, \{9\}, \{10\}, \{11\}$ & $(\{12\}, \{11\}, \{9\}, \{10\})$\\
  $A_{13}$ & $\{2\}, \{13\}$ & \\
  $A_{14}$ & $\{2\}, \{13\}$ & $(\{13,14\}, \{2\})$\\ \hline
\end{tabular}
\medskip
  \caption{The process of the map $\psi$ for each step $A_i\ne V$.}
  \label{tab:MH}
\end{table}

It is straightforward to check that the map $\psi:\MH_n\to S_n$ is a
weight-preserving map. To prove that this is a bijection we describe its
inverse.

Let \( \tau\in S_n \). Then we construct the corresponding type \( R_I \)
Meixner history \( \pi\in \MH_n \) by adding one or more steps at a time. For \(
1\le i\le n \), a block \( B \) in \( \tau \) is called an \emph{\( i
  \)-available block} if \( B \) contains an integer less than \( i \)
and the cycle containing \( B \) has an integer at least \( i \).

\begin{itemize}
\item Initially \( \pi \) is set to be the empty path.
\item For $i=1,2,\dots,n$, do the following procedure:
  \begin{description}
  \item[Case 1] \( i \) is contained in a block \( B \) and the cycle containing
    \( B \) has an integer greater than \( i \). In this case create an up step.
  \item[Case 2] \( (\{i\}) \) is a cycle in \( \tau \). Then create a
    non-labeled horizontal step.
  \item[Case 3] \( i \) is contained in a block \( B \) with \( \min(B)<i \) and
    the cycle \( \alpha\) containing \( B \) has an integer greater than \( i
    \). Let \( B_1,\dots,B_k \) be the \( i \)-available blocks with \(
    \min(B_1)<\dots<\min(B_k) \). If \( i\in B_j \), then we create a horizontal
    step labeled by \( j \).
  \item[Case 4] \( i \) is the largest integer in the cycle \( \alpha \)
    containing it and \( \alpha\ne (\{i\}) \).
    \begin{description}
    \item[Case 4-a] \( \{i\} \) is an element of the cycle \( \alpha \). Let \(
      B_1,\dots,B_k \) be the \( i \)-available blocks with \(
      \min(B_1)<\dots<\min(B_k) \). Then we can write \( \alpha= (\{i\},
      B_{s_1},\dots,B_{s_k}) \). Let \( r_1,\dots,r_k \) be the integers such
      that \( s_j \) is the \( r_j \)-th smallest integer among \(
      \{1,2,\dots,k\}\setminus\{s_1,\dots,s_{j-1}\} \). Then we create a
      non-labeled horizontal step followed by \( k \) vertical down steps
      labeled by \( r_1,\dots,r_k \).
    \item[Case 4-b] \( \{i\} \) is not an element of the cycle \( \alpha \). Let
      \( B_1,\dots,B_k \) be the \( i \)-available blocks with \(
      \min(B_1)<\dots<\min(B_k) \). Then we can write \( \alpha=
      (B_{s_1},\dots,B_{s_k}) \), where \( i\in B_{s_1} \). Let \( r_1,\dots,r_k
      \) be the integers such that \( s_j \) is the \( r_j \)-th smallest
      integer among \( \{1,2,\dots,k\}\setminus\{s_1,\dots,s_{j-1}\} \). Then we
      create a horizontal step labeled by \( 0 \) followed by \( k \) vertical
      down steps labeled by \( r_1,\dots,r_k \).
    \end{description}
  \end{description}
\end{itemize}

It is not hard to check that the above map \( \tau\mapsto \pi \) is the inverse
of $\psi$. Therefore $\psi:\MH_n\to S_n$ is a weight-preserving bijection, which
shows \eqref{eq:wt=wt}.

Finally we note that the $n$th moment of the Meixner polynomials has the
following formula due to Zeng \cite{Zeng1992} (see also de M\'edicis \cite[Theorem 2]{de_M_dicis_1998}):
\begin{equation}
    \label{eq:meixner2}
\mu_n = (1-c)^{-n} \sum_{\pi\in \Sym_n} b^{\cyc(\pi)} c^{\nexc(\pi)},
\end{equation}
where $\nexc(\pi)$ is the number of non-excedances of $\pi$, i.e., the number of
integers $i\in[n]$ with $\pi(i)<i$. Since the Meixner polynomials and the type
$R_I$ Meixner polynomials have the same $n$th moment, combining
Proposition~\ref{prop:meixner} and \eqref{eq:meixner2} yields the following
corollary.

  \begin{cor}
   We have 
\[
\sum_{\pi\in \Sym_n} b^{\cyc(\pi)} c^{\nexc(\pi)}=
\sum_{j=1}^n S(n,j) (b)_j c^j(1-c)^{n-j}.
\]
  \end{cor}

\section*{Acknowledgements}  

The authors would like to thank the anonymous referees,
who provided detailed suggestions and corrections.

\end{document}